\documentclass[a4paper]{article}

\usepackage{geometry}
\usepackage{amsmath,amssymb,amsfonts,mathtools}
\usepackage{amsthm}

\usepackage{graphicx}
\usepackage{indentfirst}
\usepackage{array}
\usepackage{mathrsfs}
\usepackage{color}
\usepackage{cite}
\usepackage[colorlinks=true]{hyperref}

\numberwithin{equation}{section}

\theoremstyle{plain}
\newtheorem{theorem}{Theorem}[section]
\newtheorem{lemma}[theorem]{Lemma}
\newtheorem{proposition}[theorem]{Proposition}

\theoremstyle{definition}
\newtheorem{definition}[theorem]{Definition}

\theoremstyle{remark}
\newtheorem{remark}[theorem]{Remark}

\newcommand{\Vol}{\operatorname{Vol}}
\newcommand{\diam}{\operatorname{diam}}

\newcommand{\dist}{\operatorname{dist}_g}
\newcommand{\KM}{K_M^{s}}           
\newcommand{\KMeps}{K_{M,\varepsilon}^{s}}
\newcommand{\Kps}{K_{p}^{s}}

\newcommand{\avg}[1]{\overline{#1}}

\begin{document}
\title{Sharp Fractional Sobolev Embeddings on Closed Manifolds}
\author{Hao Tan$^{1}$, Zetian Yan$^{2}$, Zhipeng Yang$^{1}$\thanks{Corresponding author:yangzhipeng326@163.com. }\\
        {\small $^{1}$ Department of Mathematics, Yunnan Normal University, Kunming, China.}\\
        {\small $^{2}$ Department of Mathematics, UC Santa Barbara, Santa Barbara, CA, USA}}

\date{} \maketitle

\begin{abstract}
We develop an intrinsic, heat-kernel based fractional Sobolev framework on closed Riemannian manifolds and study the critical fractional Sobolev embedding. We determine the optimal coefficient of the lower-order $L^{p}$ term and prove that the fully sharp $p$-power inequality cannot hold globally in the superquadratic range. We further establish an almost sharp inequality whose leading constant is arbitrarily close to the Euclidean best constant, and we derive improved inequalities under finitely many orthogonality constraints with respect to sign-changing test families.

\par
\smallskip
\noindent {\bf  Keywords}: Fractional Laplacian; closed Riemannian manifolds; Sobolev inequality.
		
\smallskip
\noindent {\bf MSC2020}: 35R01, 35R11, 35A15.
\end{abstract}

\tableofcontents

\section{Introduction and main results}\label{sec:intro}
Nonlocal models have become central in geometric analysis and in continuum physics.
On the geometric side, nonlocal minimal surfaces replace the classical perimeter by the $s$--perimeter
\begin{equation*}
	\mathrm{Per}_s(E;\Omega)
	=
	\iint_{(E\cap\Omega)\times(E^c)}\frac{dx\,dy}{|x-y|^{n+2s}}
	+
	\iint_{(E\setminus\Omega)\times(\Omega\setminus E)}\frac{dx\,dy}{|x-y|^{n+2s}},
	\qquad s\in(0,1).
\end{equation*}
Its first variation yields the nonlocal mean curvature
\begin{equation*}
	H_s[E](x)
	=
	c_{n,s}\,\mathrm{PV}\int_{\mathbb{R}^n}\frac{\chi_{E^c}(y)-\chi_E(y)}{|x-y|^{n+2s}}\,dy,
	\qquad x\in\partial E,
\end{equation*}
a framework initiated in \cite{Caffarelli-Roquejoffre-Savin2010} and systematically developed in \cite{Bucur-Valdinoci2016}.
On a closed Riemannian manifold $(M,g)$, it is natural to replace $|x-y|$ by the geodesic distance $d_g(x,y)$, or, in a coordinate-free spirit that fits analysis and probability on $(M,g)$, to use the heat kernel $K_M(t,x,y)$ to define intrinsic nonlocal energies \cite{Caselli-Florit-Simon-Serra-2024}.

On the physical side, anomalous diffusion and L\'evy flights lead to the fractional heat equation
\begin{equation*}
	\partial_t u + (-\Delta)^s u = 0,
	\qquad s\in(0,1),
\end{equation*}
and fractional quantum mechanics leads to the fractional Schr\"odinger equation
\begin{equation*}
	i\,\partial_t \psi = (-\Delta)^s \psi + V\psi,
\end{equation*}
see \cite{Metzler-Klafter2000,Laskin2002}. In peridynamics, nonlocal elasticity postulates bond-based interaction energies of the type
\begin{equation*}
	\mathcal{E}[u]
	=
	\frac{1}{2}\iint_{\mathbb{R}^n\times\mathbb{R}^n}
	\Phi\!\left(\frac{|u(x)-u(y)|}{|x-y|}\right)\rho_\delta(|x-y|)\,dx\,dy,
\end{equation*}
where $\rho_\delta$ is a long-range kernel \cite{Silling2000}. On manifolds, the spectral and semigroup calculus defines $(-\Delta_g)^s$ by
\begin{equation*}
	(-\Delta_g)^s u
	=
	\frac{1}{\Gamma(-s)}\int_0^\infty \bigl(e^{t\Delta_g}u-u\bigr)\frac{dt}{t^{1+s}}
	=
	\sum_{k\ge0}\lambda_k^{s}\,\langle u,\varphi_k\rangle\,\varphi_k,
\end{equation*}
linking fractional diffusion to subordinate Brownian motion and to the heat kernel; cf. \cite{Caffarelli-Silvestre2007,Stinga-Torrea2010}.

When the ambient space is a closed (compact, without boundary) Riemannian manifold $(M,g)$, the lack of translation invariance and the presence of curvature force one to rethink the very definition of fractional objects. In particular, extending the Euclidean fractional Sobolev framework to manifolds in a coordinate-free, geometrically natural way is a prerequisite for importing nonlocal tools into geometric analysis on $(M,g)$, including applications to nonlocal isoperimetry, phase transitions, and fractional curvature flows.

Let $s\in(0,1)$ and $p\in[1,\infty)$. Following the heat-kernel approach (Section~2), we define
\begin{equation}\label{eq1.1}
	K_p^s(x,y)
	=
	c_{s,p}\int_0^\infty K_M(t,x,y)\frac{dt}{t^{1+\frac{sp}{2}}},
	\qquad x\ne y,
\end{equation}
and the intrinsic seminorm
\begin{equation}\label{eq1.2}
	[u]_{W^{s,p}(M)}^{p}
	=
	\iint_{M\times M}|u(x)-u(y)|^p K_p^s(x,y)\,d\mu(x)\,d\mu(y).
\end{equation}
We set
\[
W^{s,p}(M)=\{u\in L^p(M):[u]_{W^{s,p}(M)}<\infty\}.
\]
On closed $(M,g)$ there exist constants $C_1,C_2>0$ such that for all $x\ne y$,
\[
C_1 d_g(x,y)^{-(n+sp)}\le K_p^s(x,y)\le C_2 d_g(x,y)^{-(n+sp)}.
\]
Consequently, \eqref{eq1.2} is equivalent to the geodesic Gagliardo seminorm; compare \cite{Caselli-Florit-Simon-Serra-2024,kreuml-Mordhorst2019,Rey-Saintier2024}.

Within this framework, we show that $W^{s,p}(M)$ is a Banach space. Moreover, it is separable for $1\le p<\infty$ and reflexive for $1<p<\infty$. It also satisfies a fractional Poincar\'e inequality and Sobolev-type embedding results (see Section~4.1).

On the other hand, the sharp fractional Sobolev inequality for the quadratic case $p=2$ in the Euclidean setting reads
\begin{equation}\label{eq1.3}
	\|u\|_{L^{2_s^*}(\mathbb{R}^n)}^2
	\le
	K(n,s,2)\iint_{\mathbb{R}^n\times\mathbb{R}^n}\frac{|u(x)-u(y)|^2}{|x-y|^{n+2s}}\,dx\,dy,
	\qquad
	2_s^*=\frac{2n}{n-2s},
\end{equation}
with the sharp constant $K(n,s,2)$, attained by the standard fractional bubbles; see \cite{DiNezza-Palatucci-Valdinoci2012}.
For local (gradient) inequalities on manifolds, the foundational works of Aubin, Hebey, Druet and Bakry established Euclidean sharpness of the leading constant, the structure and closure of the best lower-order term, and orthogonality improvements; see \cite{Aubin1976,Hebey1999,Druet1998,Bakry1994}.
In the fractional setting on compact manifolds, recent contributions include intrinsic characterizations of $W^{s,p}(M)$ and nonlocal inequalities for equations on $(M,g)$; see \cite{kreuml-Mordhorst2019,Rey-Saintier2024,Caselli-Florit-Simon-Serra-2024}.

We investigate optimal fractional Sobolev embeddings on closed $(M,g)$ in an intrinsic framework.
Assume $sp<n$ and set
\[
p_s^*=\frac{np}{n-sp}.
\]
For $u\in W^{s,p}(M)$ we consider the two standard formulations
\begin{equation}\label{eq1.4}
\|u\|_{L^{p_s^*}(M)}\le A[u]_{W^{s,p}(M)}+B\|u\|_{L^p(M)},
\end{equation}
and
\begin{equation}\label{eq1.4p}
\|u\|_{L^{p_s^*}(M)}^{p}\le \widetilde A[u]_{W^{s,p}(M)}^{p}+\widetilde B\|u\|_{L^p(M)}^{p}.
\end{equation}
We define $\beta_p(M)$ and $\overline\beta_p(M)$ as the best (infimal) constants $B$ and $\widetilde B$ in \eqref{eq1.4} and \eqref{eq1.4p}, respectively, in the spirit of Hebey \cite{Hebey1996}.
We refer to this as the $\mathcal B$--program. We also study the $\mathcal A$--program, namely the sharp leading constant and its improvement under orthogonality constraints.

Our analysis yields complete answers for the $\mathcal B$--program and sharp leading constants for the $\mathcal A$--program on closed manifolds.

\begin{theorem}\label{thm1.1}
\medskip\noindent\textbf{The $\mathcal B$--program.}
\begin{itemize}
	\item[(B1)] If $n>sp$ and $1\le p<\infty$, then
	\[
	\beta_p(M)=\Vol(M)^{-s/n},
	\qquad
	\overline\beta_p(M)=\Vol(M)^{-sp/n}.
	\]
	\item[(B2)] For the linear form in \eqref{eq1.4}, the set of admissible constants $B$ is closed at its infimum, and the optimal inequality holds with $B=\beta_p(M)$.
	\item[(B3)] If $n\ge2$, the $p$-power optimal inequality holds for every $p\in[1,2]$ with $sp<n$. When $n\ge3$ and $p\in(2,n)$, it may fail in general.
\end{itemize}
\end{theorem}

\begin{theorem}\label{thm1.2}
\medskip\noindent\textbf{The $\mathcal A$--program.}
Assume $n>sp$ and set $p_s^*=\frac{np}{n-sp}$.
\begin{itemize}
	\item[(A1)] For every $\varepsilon>0$ there exists $B_\varepsilon$ such that
	\[
	\|u\|_{L^{p_s^*}(M)}^p
	\le
	\bigl(K(n,s,p)+\varepsilon\bigr)[u]_{W^{s,p}(M)}^p
	+
	B_\varepsilon\|u\|_{L^p(M)}^p,
	\qquad u\in W^{s,p}(M),
	\]
	where $K(n,s,p)$ is the Euclidean best constant. In particular, the leading constant is Euclidean-sharp on any closed $(M,g)$.
	\item[(A2)] Let $f_i\in C^1(M)$, $i=1,\ldots,N$, be sign-changing functions such that
	\[
	\sum_{i=1}^N |f_i|^p\equiv 1 \quad \text{on } M.
	\]
	If, in addition, $u$ satisfies the orthogonality conditions
	\[
	\int_M f_i|f_i|^{p_s^*-1}|u|^{p_s^*}\,d\mu=0,
	\qquad i=1,\ldots,N,
	\]
	then the leading constant improves by the factor $2^{-sp/n}$: for every $\varepsilon>0$ there exists $B_{\varepsilon,\{f_i\}}$ such that
	\[
	\|u\|_{L^{p_s^*}(M)}^p
	\le
	\Bigl(\frac{K(n,s,p)}{2^{sp/n}}+\varepsilon\Bigr)[u]_{W^{s,p}(M)}^{p}
	+
	B_{\varepsilon,\{f_i\}}\|u\|_{L^p(M)}^p.
	\]
\end{itemize}
\end{theorem}

\begin{remark}\label{rem1.3}
The case $p=2$ corresponds to the main result of \cite{Rey-Saintier2024}. In that work, the authors obtain an almost sharp inequality by following the strategy of Aubin \cite{Aubin1976}, which relies on the classification of extremals in $\mathbb{R}^n$. Namely, up to scaling and translation, the standard fractional bubbles can be written as
\[
U_{\varepsilon,x_0}(x)
=
C_{n,s}\left(\frac{\varepsilon}{\varepsilon^2+|x-x_0|^2}\right)^{\frac{n-2s}{2}}.
\]
For $p\ne 2$, explicit formulas and a complete classification of extremals for the Euclidean sharp constant are not available in general.
In contrast, the approach in \cite{Hang2022,Yan2023} (and in the present work) is based on the concentration--compactness principle, which avoids the need for a full characterization of optimizers in $\mathbb{R}^n$.
\end{remark}

Theorems \ref{thm1.1} and \ref{thm1.2} describe how the geometry of a closed manifold influences fractional Sobolev embeddings: the leading nonlocal term is Euclidean in nature, while the manifold enters through the best lower-order $L^p$ term and through orthogonality constraints. The sharpness and closure properties extend the local manifold theory \cite{Aubin1976,Druet1998,Hebey1996,Bakry1994} to the fractional regime, complementing recent advances on nonlocal equations and inequalities on compact manifolds \cite{Rey-Saintier2024}. It is worth noting that in dimension $n=2$, the $p$-power optimal inequality in the $\mathcal B$ program holds on the range $p\in[1,2]$ in the fractional regime $s\in(0,1)$, thereby including the endpoint $p=2$. By contrast, in the local case the corresponding statement is valid only for $p\in[1,2)$.

Section~2 recalls basic geometric and heat-kernel facts on closed manifolds. Section~3 reviews three equivalent definitions of $(-\Delta_g)^s$ (spectral, semigroup/singular integral, and extension). Section~4 develops the intrinsic spaces $W^{s,p}(M)$, proves their core properties, and carries out the $\mathcal B$ and $\mathcal A$ programs stated above.

\section{Preliminaries on closed Riemannian manifolds}

In this section, we collect several elementary facts that will be used in the main estimates of the paper. For background on Riemannian geometry, we refer the reader to Chavel~\cite{Chavel1993}, do Carmo~\cite{DoCarmo1992}, Hebey~\cite{Hebey1996}, and Jost~\cite{Jost1995}. Standard references on the heat kernel include the monograph~\cite{Davies1989} and the survey~\cite{GrigorYan1999}.

\subsection{Laplace operator and eigenvalues}

In $\mathbb{R}^n$ the Euclidean Laplacian $\Delta$ acts by
\[
\Delta u=\sum_{i=1}^n \frac{\partial^2 u}{\partial (x^i)^2}.
\]
On a Riemannian manifold $(M^n,g)$, the Laplace--Beltrami operator $\Delta_g$ is given in local coordinates by
\begin{equation}\label{eq2.1}
\Delta_g u
=
\frac{1}{\sqrt{|g|}}
\sum_{i,j=1}^n \frac{\partial}{\partial x^i}
\left(\sqrt{|g|}\,g^{ij}\,\frac{\partial u}{\partial x^j}\right),
\qquad
|g|=\det(g_{ij}),
\qquad
(g^{ij})=(g_{ij})^{-1}.
\end{equation}
Throughout we adopt the sign convention that $\Delta_g\le 0$ on $L^2(M)$, so that $-\Delta_g$ is a nonnegative self-adjoint operator.

The Riemannian volume measure associated to $g$ is
\[
d\mu=\sqrt{|g|}\,dx^1\cdots dx^n,
\qquad |g|=\det(g_{ij}).
\]
By integration by parts, on a closed manifold $M^n$, we have
\begin{equation}\label{eq2.2}
\int_M v\,\Delta_g u\,d\mu
=
-\int_M \langle \nabla_g u,\nabla_g v\rangle_g\,d\mu
\end{equation}
for all $u,v\in C^\infty(M)$.

Consider the eigenvalue problem for $\Delta_g$
\[
-\Delta_g u=\lambda u,
\qquad u\in C^\infty(M).
\]
Taking $v\equiv 1$ in \eqref{eq2.2} gives
\[
\int_M \Delta_g u\,d\mu=0.
\]
Integrating the eigenvalue equation over $M$ yields
\[
0=\int_M \Delta_g u\,d\mu=-\lambda\int_M u\,d\mu,
\]
so if $\lambda\ne 0$ then $\int_M u\,d\mu=0$. Moreover, taking $v=u$ in \eqref{eq2.2} we obtain
\[
\lambda\int_M u^2\,d\mu
=
-\int_M u\,\Delta_g u\,d\mu
=
\int_M |\nabla_g u|_g^2\,d\mu
\ge 0,
\]
showing that all eigenvalues are nonnegative. Taking $u$ constant yields $\lambda_0=0$.

By the spectral theory of elliptic operators on closed manifolds, there exists an orthonormal basis $\{\varphi_k\}_{k=0}^{\infty}$ of $L^2(M)$ consisting of eigenfunctions of $-\Delta_g$ with eigenvalues
\[
0=\lambda_0<\lambda_1\le\lambda_2\le\cdots,\qquad \lambda_k\to+\infty.
\]
For $u\in L^2(M)$, writing $u=\sum_{k\ge0}u_k\varphi_k$ with $u_k=\langle u,\varphi_k\rangle_{L^2(M)}$, the heat semigroup satisfies, for every $t\ge 0$,
\begin{equation}\label{eq2.3}
e^{t\Delta_g}u=\sum_{k\ge0} e^{-t\lambda_k}u_k\,\varphi_k.
\end{equation}
In particular, $e^{t\Delta_g}$ is a bounded self-adjoint operator on $L^2(M)$.

\subsection{Heat kernels on closed Riemannian manifolds}
The purpose of this subsection is to give a brief introduction to heat kernels on a closed Riemannian manifold $(M,g)$.
We start from the Euclidean case. On $\mathbb{R}^n$, the heat kernel $p(t,x,y)$ is the fundamental solution to
\[
\partial_t u=\Delta u,
\qquad
u(0,\cdot)=\delta_y,
\]
and it is given explicitly by
\[
p(t,x,y)=\frac{1}{(4\pi t)^{n/2}}\exp\Big(-\frac{|x-y|^2}{4t}\Big),
\qquad t>0,\ x,y\in\mathbb{R}^n.
\]
Equivalently, for bounded continuous $f$, the Cauchy problem
\[
\partial_t u=\Delta u,
\qquad
u(0,x)=f(x),
\]
has the solution
\[
u(t,x)=\int_{\mathbb{R}^n} p(t,x,y)f(y)\,dy.
\]

On a closed manifold $(M,g)$, let $f\in L^2(M)$ and consider the initial value problem
\begin{equation}\label{eq2.4}
\begin{cases}
\partial_t u=\Delta_g u,\\
u(0,x)=f(x), \qquad x\in M,\ t>0.
\end{cases}
\end{equation}
We interpret \eqref{eq2.4} in the semigroup (mild) sense by setting
\begin{equation}\label{eq2.5}
u(t)=e^{t\Delta_g}f,
\qquad t\ge 0.
\end{equation}
Then $u\in C([0,\infty);L^2(M))$. Moreover, for every $t>0$ one has $u(t)\in \mathrm{Dom}(\Delta_g)$ and
$u\in C^1((0,\infty);L^2(M))$ with $\dot u(t)=\Delta_g u(t)$ in $L^2(M)$.

Rewriting $f$ in terms of an $L^2$-orthonormal basis $\{\varphi_k\}_{k\ge0}$, we have
\[
f=\sum_{k=0}^\infty a_k\varphi_k,
\qquad
a_k=\langle f,\varphi_k\rangle_{L^2(M)}.
\]
Using \eqref{eq2.3}, for $t>0$ we obtain
\[
\begin{aligned}
e^{t\Delta_g}f(x)
&=\sum_{k=0}^\infty a_k e^{-t\lambda_k}\varphi_k(x)\\
&=\sum_{k=0}^\infty e^{-t\lambda_k}\varphi_k(x)\int_M f(y)\overline{\varphi_k(y)}\,d\mu(y).
\end{aligned}
\]
For every $t>0$, the series
\[
\sum_{k=0}^\infty e^{-t\lambda_k}\varphi_k(x)\overline{\varphi_k(y)}
\]
converges in $C^\infty(M\times M)$ and defines the heat kernel $K_M(t,x,y)$. Therefore,
\[
e^{t\Delta_g}f(x)=\int_M K_M(t,x,y)f(y)\,d\mu(y),
\qquad t>0,
\]
and the mild solution to \eqref{eq2.4} is
\begin{equation}\label{eq2.6}
u(t,x)=\int_M K_M(t,x,y)f(y)\,d\mu(y),
\qquad t>0,
\end{equation}
which is another form of \eqref{eq2.5}. Hence $e^{t\Delta_g}$ admits the integral kernel $K_M(t,x,y)$.

\begin{proposition}\label{pro2.1}
The heat kernel $K_M$ satisfies the following properties.
\begin{itemize}
\item[(1)]For each fixed $y\in M$, the function $(t,x)\mapsto K_M(t,x,y)$ is smooth on $(0,\infty)\times M$ and solves
\[
\partial_t K_M(t,x,y)=\Delta_{g,x}K_M(t,x,y),
\qquad t>0,\ x\in M.
\]
Moreover, for every $\psi\in C(M)$,
\[
\int_M K_M(t,x,y)\psi(x)\,d\mu(x)\to \psi(y)
\qquad \text{as }t\to 0^+.
\]

\item[(2)]For all $t,s>0$ and all $x,y\in M$,
\[
K_M(t+s,x,y)=\int_M K_M(t,x,z)K_M(s,z,y)\,d\mu(z).
\]

\item[(3)]
For all $t>0$ and $x,y\in M$,
\[
K_M(t,x,y)=K_M(t,y,x).
\]
\end{itemize}
\end{proposition}

\begin{proposition}\label{pro2.2}
Let $(M,g)$ be a closed Riemannian manifold. Then:
\begin{itemize}
\item[(1)] The heat semigroup preserves constants:
\[
e^{t\Delta_g}1=1
\qquad \text{for all }t\ge 0.
\]

\item[(2)] The heat kernel has unit mass: for every $t>0$ and every $x\in M$,
\[
\int_M K_M(t,x,y)\,d\mu(y)=1.
\]

\item[(3)] Short-time Gaussian bounds.
There exist $t_0>0$ and constants $c,C>0$ depending only on $(M,g)$ such that for all $0<t\le t_0$ and all $x,y\in M$,
\[
\frac{c}{t^{n/2}}\exp\Big(-\frac{d_g(x,y)^2}{Ct}\Big)
\le
K_M(t,x,y)
\le
\frac{C}{t^{n/2}}\exp\Big(-\frac{d_g(x,y)^2}{ct}\Big).
\]

\item[(4)] Large-time behavior.
Let $\lambda_1>0$ be the first nonzero eigenvalue of $-\Delta_g$. Then there exists $C>0$ such that for all $t\ge 1$ and all $x,y\in M$,
\[
\Bigl|K_M(t,x,y)-\frac{1}{\Vol(M)}\Bigr|\le C e^{-\lambda_1 t}.
\]
In particular,
\[
0<K_M(t,x,y)\le \frac{1}{\Vol(M)}+C e^{-\lambda_1 t}.
\]
\end{itemize}
\end{proposition}

\section{Fractional Laplacian on closed Riemannian manifolds}
Throughout this section, unless explicitly stated otherwise, $(M^n,g)$ denotes a closed $n$-dimensional Riemannian manifold.
Motivated by the Euclidean constructions in \cite{Caffarelli-Silvestre2007,DiNezza-Palatucci-Valdinoci2012,Stinga-Torrea2010}, we present several equivalent definitions of the fractional Laplacian $(-\Delta_g)^s$ for $s\in(0,1)$.

For $u\in L^2(M)$ we write its spectral expansion
\[
u=\sum_{k=0}^\infty u_k \phi_k,
\qquad
u_k=\langle u,\phi_k\rangle_{L^2(M)}=\int_M u\,\overline{\phi_k}\,d\mu.
\]
For $s\ge 0$ we define
\begin{equation}\label{eq3.1}
H^s(M)
=
\bigg\{
u=\sum_{k=0}^\infty u_k\phi_k\in L^2(M)
\ \Big|\
\sum_{k=0}^\infty (1+\lambda_k)^{s}|u_k|^2<\infty
\bigg\},
\end{equation}
endowed with the norm
\[
\|u\|_{H^s(M)}^2
=
\sum_{k=0}^\infty (1+\lambda_k)^s|u_k|^2,
\]
which is equivalent to the standard Sobolev $H^s$ norm on closed manifolds.
In particular, $\|u\|_{L^2(M)}\le \|u\|_{H^s(M)}$ for all $s\ge 0$.

\begin{definition}[Spectral fractional Laplacian]\label{def3.1}
Let $s\in(0,1)$.
\begin{itemize}
\item[(i)] As an unbounded operator on $L^2(M)$.
Its domain is
\[
\mathrm{Dom}((-\Delta_g)^s)
=
\bigg\{
u=\sum_{k=0}^\infty u_k\phi_k\in L^2(M)
\ \Big|\
\sum_{k=0}^\infty \lambda_k^{2s}|u_k|^2<\infty
\bigg\},
\]
and for $u\in \mathrm{Dom}((-\Delta_g)^s)$,
\[
(-\Delta_g)^s u
=
\sum_{k=0}^\infty \lambda_k^s u_k\phi_k
\in L^2(M).
\]
Moreover, $\mathrm{Dom}((-\Delta_g)^s)=H^{2s}(M)$ as sets, with equivalent norms.

\item[(ii)] As a bounded operator $H^s(M)\to H^{-s}(M)$.
For $u=\sum u_k\phi_k\in H^s(M)$ we define $(-\Delta_g)^s u\in H^{-s}(M)$ by duality:
\begin{equation}\label{eq3.2}
\langle (-\Delta_g)^s u,\psi\rangle
=
\sum_{k=0}^\infty \lambda_k^s u_k\overline{\psi_k},
\qquad
\psi=\sum_{k=0}^\infty \psi_k\phi_k\in H^s(M).
\end{equation}
This defines a continuous pairing on $H^s(M)\times H^s(M)$.
\end{itemize}
\end{definition}

\begin{remark}\label{rem3.2}
It is often convenient to use the quadratic form
\[
|u|_{H^s_\Delta}^2
=
\sum_{k=0}^\infty \lambda_k^{s}|u_k|^2
=
\|(-\Delta_g)^{s/2}u\|_{L^2(M)}^2,
\]
which is a seminorm (it vanishes on constants).
On mean-zero functions, $|\,\cdot\,|_{H^s_\Delta}$ is equivalent to the full $H^s(M)$ norm.
\end{remark}

\subsection{Heat semigroup and singular integral}

The fractional Laplacian $(-\Delta_g)^s$ can be defined as the $s$-th power (in the spectral sense) of the Laplace--Beltrami operator on a closed Riemannian manifold, and it admits a semigroup representation.

\begin{definition}\label{def3.3}
Let $s\in(0,1)$ and let $u\in H^{s}(M)$ with spectral coefficients $u_k=\langle u,\phi_k\rangle_{L^2(M)}$.
We fix the normalization constant
\[
c_s=\frac{1}{|\Gamma(-s)|},
\]
and define
\begin{equation}\label{eq3.3}
\begin{aligned}
(-\Delta_g)^s u
&=\sum_{k=0}^{\infty} \lambda_k^{s}u_k\phi_k
\qquad \text{in } L^2(M)\ \text{if }u\in H^{2s}(M),\\[1mm]
\langle (-\Delta_g)^s u,\psi\rangle
&=c_s\int_0^{\infty}\langle u-e^{t\Delta_g}u,\psi\rangle_{L^2(M)}\frac{dt}{t^{1+s}}
\qquad \text{for all }\psi\in H^{s}(M),
\end{aligned}
\end{equation}
which defines $(-\Delta_g)^s u\in H^{-s}(M)$ in general.
\end{definition}

The equivalence of the two expressions follows from the scalar identity
\[
\lambda^{s}
=c_s\int_0^{\infty}\bigl(1-e^{-\lambda t}\bigr)\frac{dt}{t^{1+s}},
\qquad \lambda>0,\ s\in(0,1),
\]
together with the spectral expansions $u=\sum u_k\phi_k$ and $\psi=\sum \psi_k\phi_k$.
Indeed, for $u,\psi\in H^{s}(M)$,
\begin{equation}\label{eq3.4}
\begin{aligned}
\langle (-\Delta_g)^s u,\psi\rangle
&=\sum_{k=0}^{\infty} \lambda_k^{s}u_k\overline{\psi_k} \\
&=c_s\sum_{k=0}^{\infty}\int_0^{\infty}\bigl(1-e^{-t\lambda_k}\bigr)u_k\overline{\psi_k}\frac{dt}{t^{1+s}} \\
&=c_s\int_0^{\infty}\left(\sum_{k=0}^{\infty}u_k\overline{\psi_k}-\sum_{k=0}^{\infty}e^{-t\lambda_k}u_k\overline{\psi_k}\right)\frac{dt}{t^{1+s}} \\
&=c_s\int_0^{\infty}\langle u-e^{t\Delta_g}u,\psi\rangle_{L^2(M)}\frac{dt}{t^{1+s}}.
\end{aligned}
\end{equation}

The interchange of sum and integral is justified by absolute integrability.
For $s\in(0,1)$ and $\lambda\ge 0$,
\[
\int_0^\infty \frac{1-e^{-\lambda t}}{t^{1+s}}\,dt
=\frac{\lambda^{s}}{c_s}.
\]
Therefore,
\[
\sum_{k=0}^\infty \lambda_k^{s}|u_k||\psi_k|
\le
\left(\sum_{k=0}^\infty \lambda_k^{s}|u_k|^2\right)^{1/2}
\left(\sum_{k=0}^\infty \lambda_k^{s}|\psi_k|^2\right)^{1/2}
<\infty,
\]
for all $u,\psi\in H^{s}(M)$.

\begin{theorem}\label{thm3.4}
Let $u,\psi\in H^{s}(M)$ with $s\in(0,1)$. Then
\begin{equation}\label{eq3.5}
\big\langle (-\Delta_g)^s u,\psi\big\rangle_{H^{-s},H^{s}}
=
\frac12\int_{M}\!\int_{M}\bigl(u(x)-u(y)\bigr)\overline{\bigl(\psi(x)-\psi(y)\bigr)}
\,\KM(x,y)\,d\mu(x)\,d\mu(y),
\end{equation}
where the kernel $\KM$ is given, for $x\ne y$, by
\begin{equation}\label{eq3.6}
0\le \KM(x,y)
=
c_s\int_{0}^{\infty} K_M(t,x,y)\,\frac{dt}{t^{1+s}},
\qquad
c_s=\frac{1}{|\Gamma(-s)|},
\end{equation}
and there exists $C_{M,s}>0$ depending only on $(M,g)$ and $s$ such that
\[
\KM(x,y)\le \frac{C_{M,s}}{\dist(x,y)^{\,n+2s}},
\qquad x\ne y.
\]
\end{theorem}

\begin{proof}
By Definition~\ref{def3.3} (with $c_s=\frac{1}{|\Gamma(-s)|}$), for $u,\psi\in H^s(M)$,
\begin{equation}\label{eq3.4bis}
\big\langle (-\Delta_g)^s u,\psi\big\rangle_{H^{-s},H^{s}}
=
c_s\int_{0}^{\infty}\big\langle u-e^{t\Delta_g}u,\psi\big\rangle_{L^2(M)}\,\frac{dt}{t^{1+s}}.
\end{equation}

Fix $t>0$. Using the heat-kernel representation \eqref{eq2.5} and Proposition~\ref{pro2.2}, we write
\begin{align*}
\big\langle u-e^{t\Delta_g}u,\psi\big\rangle_{L^2(M)}
&=\int_M u(x)\overline{\psi(x)}\,d\mu(x)-\int_M\!\int_M K_M(t,x,y)u(y)\overline{\psi(x)}\,d\mu(y)d\mu(x)\\
&=\int_M\!\int_M K_M(t,x,y)u(x)\overline{\psi(x)}\,d\mu(y)d\mu(x)
   -\int_M\!\int_M K_M(t,x,y)u(y)\overline{\psi(x)}\,d\mu(y)d\mu(x)\\
&=\int_M\!\int_M K_M(t,x,y)\bigl(u(x)-u(y)\bigr)\overline{\psi(x)}\,d\mu(x)d\mu(y).
\end{align*}
By symmetry $K_M(t,x,y)=K_M(t,y,x)$, exchanging $x$ and $y$ yields also
\[
\big\langle u-e^{t\Delta_g}u,\psi\big\rangle_{L^2(M)}
=-\int_M\!\int_M K_M(t,x,y)\bigl(u(x)-u(y)\bigr)\overline{\psi(y)}\,d\mu(x)d\mu(y).
\]
Adding the two identities gives
\begin{equation}\label{eq3.4ter}
2\big\langle u-e^{t\Delta_g}u,\psi\big\rangle_{L^2(M)}
=
\int_M\!\int_M K_M(t,x,y)\bigl(u(x)-u(y)\bigr)\overline{\bigl(\psi(x)-\psi(y)\bigr)}\,d\mu(x)d\mu(y).
\end{equation}

Insert \eqref{eq3.4ter} into \eqref{eq3.4bis}. By Fubini theorem we obtain
\[
\big\langle (-\Delta_g)^s u,\psi\big\rangle_{H^{-s},H^{s}}
=
\frac12\int_M\!\int_M \bigl(u(x)-u(y)\bigr)\overline{\bigl(\psi(x)-\psi(y)\bigr)}
\left(c_s\int_0^\infty K_M(t,x,y)\frac{dt}{t^{1+s}}\right)
d\mu(x)d\mu(y),
\]
which is exactly \eqref{eq3.5} with \eqref{eq3.6}.

We now estimate $\KM(x,y)$ for $x\ne y$. Using the short-time Gaussian upper bound in Proposition~\ref{pro2.2}, there exist $t_0>0$ and constants $C,c>0$ such that for $0<t\le t_0$,
\[
0\le K_M(t,x,y)\le \frac{C}{t^{n/2}}\exp\Big(-\frac{\dist(x,y)^2}{c\,t}\Big).
\]
Hence
\[
\int_{0}^{t_0} K_M(t,x,y)\,\frac{dt}{t^{1+s}}
\le
C\int_{0}^{t_0} t^{-\frac{n}{2}-1-s}\exp\Big(-\frac{\dist(x,y)^2}{c\,t}\Big)\,dt
\le
\frac{C}{\dist(x,y)^{\,n+2s}}.
\]
For the large-time part, since $M$ is compact, $K_M(t,x,y)$ is bounded uniformly in $(x,y)$ for $t\ge t_0$, and therefore
\[
\int_{t_0}^{\infty} K_M(t,x,y)\,\frac{dt}{t^{1+s}}\le C(M,s).
\]
Since $\dist(x,y)\le \diam(M)$ for all $x,y\in M$, we have
\[
C(M,s)\le \frac{C(M,s)\,\diam(M)^{n+2s}}{\dist(x,y)^{\,n+2s}},
\]
so combining the two ranges yields
\[
\int_{0}^{\infty} K_M(t,x,y)\,\frac{dt}{t^{1+s}}
\le
\frac{C_{M,s}}{\dist(x,y)^{\,n+2s}},
\qquad x\ne y,
\]
and multiplying by $c_s$ proves the bound in \eqref{eq3.6}.
\end{proof}

\begin{remark}\label{rem3.5}
On a noncompact manifold, the identity $\int_M K_M(t,x,y)\,d\mu(y)=1$ may fail in general.
It holds, for example, under stochastic completeness. If mass conservation fails, an additional term appears in the derivation above, just as in the fractional divergence-form operators studied in \cite{Caffarelli-Stinga2016}.
\end{remark}

Based on Theorem~\ref{thm3.4}, we can now give another definition of the fractional Laplacian on $M$, closely related to the spectral one, which expresses it as a singular integral operator.

\begin{definition}\label{def3.6}
Let $s\in(0,1)$.
For $u\in C^\infty(M)$, the fractional Laplacian $(-\Delta_g)^s$ is defined by
\[
\begin{aligned}
(-\Delta_g)^{s} u(x)
    &= \mathrm{p.v.}\!\int_M (u(x)-u(y))\,K_M^s(x,y)\,d\mu(y) \\
    &= \lim_{\varepsilon\to 0}\int_M (u(x)-u(y))\,K^s_{M,\varepsilon}(x,y)\,d\mu(y).
\end{aligned}
\]
Here $\mathrm{p.v.}$ denotes the principal value, as encoded by the limiting procedure above.
The kernel $K_M^s(x,y)$ is the singular kernel introduced in \eqref{eq3.6}, and
\[
K^s_{M,\varepsilon}(x,y)
    = c_s
      \int_0^{\infty} K_M(t,x,y)\,
      e^{-\frac{\varepsilon^2}{4t}}
      \frac{dt}{t^{1+s}}
\]
is a natural regularization, where $c_s$ is the same constant as in Definition~\ref{def3.3}.
\end{definition}

\begin{remark}\label{rem3.7}
If the closed manifold $M$ is replaced by Euclidean space $\mathbb{R}^n$, then
\[
\begin{aligned}
K_M^{s}(x,y)
&= c_s \int_0^{\infty} K_{\mathbb{R}^n}(t,x,y)\,\frac{dt}{t^{1+s}} \\
&= c_s\int_0^{\infty}
        \left(\frac{1}{(4\pi t)^{n/2}}
               e^{-\frac{|x-y|^2}{4t}}\right)
        \frac{dt}{t^{1+s}} \\
&= \frac{\alpha_{n,s}}{|x-y|^{\,n+2s}},
\end{aligned}
\]
where
\[
\alpha_{n,s}
    = \frac{2^{2s}\,\Gamma\!\left(\frac{n+2s}{2}\right)}
           {\pi^{n/2}\,|\Gamma(-s)|}.
\]
Thus we recover the classical fractional Laplacian kernel on $\mathbb{R}^n$.

Moreover, when $M=\mathbb{R}^n$,
\[
K^s_{\mathbb{R}^n,\varepsilon}(x,y)
= c_s\int_0^{\infty}
        K_{\mathbb{R}^n}(t,x,y)\,e^{-\frac{\varepsilon^2}{4t}}
        \frac{dt}{t^{1+s}}
= \frac{\alpha_{n,s}}{\left(|x-y|^2+\varepsilon^2\right)^{\frac{n+2s}{2}}},
\]
which is a very natural regularization of $\alpha_{n,s}\,|x-y|^{-(n+2s)}$.
It is straightforward to verify that this regularization yields the same principal value as integrating over
$\mathbb{R}^n\setminus B_\varepsilon(x)$ and letting $\varepsilon\to 0^{+}$.

The same holds on a Riemannian manifold: many regularizations of the singular kernel $K_M^s(x,y)$ lead to the same principal value under mild assumptions, as shown in \cite[Proposition~2.5]{Caselli-Florit-Simon-Serra-2024}.
\end{remark}

\begin{theorem}\label{thm3.8}
	For every $s\in(0,1)$, Definitions~\ref{def3.3} and \ref{def3.6} agree:
	\begin{itemize}
		\item[(1)] If $u\in C^\infty(M)$, the two definitions coincide pointwise everywhere.
		\item[(2)] If $u\in L^2(M)$, they coincide in the sense of distributions.
	\end{itemize}
\end{theorem}

\begin{proof}
	The argument follows \cite{Caselli-Florit-Simon-Serra-2024}; we include it for completeness.
	
	\textbf{Step 1.}
	Let $u\in C^\infty(M)$ and $\varepsilon>0$. By Proposition~\ref{pro2.2} and the heat kernel representation,
	\begin{equation}\label{eq3.8}
		c_s\int_{0}^{\infty}\!\big(u-e^{t\Delta_g}u\big)(x)\,e^{-\frac{\varepsilon^2}{4t}}\,
		\frac{dt}{t^{1+s}}
		=\int_M \big(u(x)-u(y)\big)\,\KMeps(x,y)\,d\mu(y),
	\end{equation}
	where
	\[
	\KMeps(x,y)=c_s\int_{0}^{\infty}K_M(t,x,y)\,e^{-\frac{\varepsilon^2}{4t}}\,
	\frac{dt}{t^{1+s}}.
	\]
	Since $u$ is smooth, letting $\varepsilon\rightarrow0$ gives
	\[
	c_s\int_{0}^{\infty}\!\big(u-e^{t\Delta_g}u\big)(x)\,\frac{dt}{t^{1+s}}
	=\mathrm{p.v.}\!\int_M \big(u(x)-u(y)\big)\,\KM(x,y)\,d\mu(y),
	\]
	which proves (1).
	
	\textbf{Step 2.}
	Let $u\in L^2(M)$ and $\varphi\in C^\infty(M)$. Multiply \eqref{eq3.8} by $u(x)$ and integrate over $M$ to get
	\begin{equation}\label{eq3.9}
		\begin{aligned}
		&c_s\int_M\!\int_{0}^{\infty}\!\big(\varphi-e^{t\Delta_g}\varphi\big)(x)\,u(x)\,
		e^{-\frac{\varepsilon^2}{4t}}\frac{dt}{t^{1+s}}\,d\mu(x)\\
		&=\iint_{M\times M}\big(\varphi(x)-\varphi(y)\big)\,u(x)\,\KMeps(x,y)\,d\mu(x)\,d\mu(y).
		\end{aligned}
	\end{equation}
	For fixed $\varepsilon>0$ both sides are absolutely convergent, so we may exchange the order of integration. Using the self-adjointness of $e^{t\Delta_g}$ in $L^2(M)$,
	\[
	c_s\int_{0}^{\infty}\!e^{-\frac{\varepsilon^2}{4t}}\frac{dt}{t^{1+s}}\,
	\langle u,\varphi-e^{t\Delta_g}\varphi\rangle_{L^2}
	=
	c_s\int_{0}^{\infty}\!e^{-\frac{\varepsilon^2}{4t}}\frac{dt}{t^{1+s}}\,
	\langle u-e^{t\Delta_g}u,\varphi\rangle_{L^2}.
	\]
	
	On the right-hand side of \eqref{eq3.9}, by the symmetry $\KMeps(x,y)=\KMeps(y,x)$,
	\[
	\iint_{M\times M}\big(\varphi(x)-\varphi(y)\big)\,u(x)\,\KMeps(x,y)\,d\mu(x)\,d\mu(y)
	=\int_M\!\bigg[\int_M\big(\varphi(x)-\varphi(y)\big)\,\KMeps(x,y)\,d\mu(y)\bigg]\,u(x)\,d\mu(x).
	\]
	Letting $\varepsilon\rightarrow0$ and invoking part (1) for the test function $\varphi$ yields
	\[
	\langle u,(-\Delta_g)^s\varphi\rangle_{L^2}
	=\int_M \bigg[\mathrm{p.v.}\!\int_M\big(\varphi(x)-\varphi(y)\big)\,\KM(x,y)\,d\mu(y)\bigg] u(x)\,d\mu(x),
	\]
	i.e. Definitions~\ref{def3.3} and \ref{def3.6} agree in the sense of distributions. This proves (2).
\end{proof}

\subsection{Dirichlet-to-Neumann map via an extension problem}
We first relate the heat semigroup to the extension problem \eqref{eq3.10} and obtain the corresponding Poisson formula.
\begin{definition}\label{def3.9}
Let $s\in(0,1)$ and $f\in H^s(M)$.
Define $U : M\times(0,\infty)\rightarrow\mathbb{R}$ by
\begin{equation}\label{eq3.11}
U(x,y)
    = \frac{y^{2s}}{2^{\,2s}\Gamma(s)}
      \int_{0}^{\infty} \big(e^{t\Delta_g}f\big)(x)\,
      e^{-\frac{y^2}{4t}}\,
      \frac{dt}{t^{1+s}}.
\end{equation}
Then $U$ solves
\begin{equation}\label{eq3.10}
\begin{cases}
\Delta_g U(x,y)
    + \dfrac{1-2s}{y}\,\partial_y U(x,y)
    + \partial_{yy}U(x,y) = 0,
    & x\in M,\ y>0,\\[4pt]
U(x,0)=f(x),
\end{cases}
\end{equation}
and the fractional Laplacian is realized as the Dirichlet--to--Neumann operator
\[
(-\Delta_g)^{s} f(x)
    =-c(s)\,\lim_{y\to 0^+} y^{1-2s}\,\partial_y U(x,y),
\qquad
c(s)=\frac{2^{\,2s-1}\Gamma(s)}{\Gamma(1-s)}.
\]
\end{definition}

\begin{lemma}\label{lem3.10}
Let $s\in(0,1)$ and $f\in C(M)$.  
The function $U$ defined by \eqref{eq3.11} solves the extension problem \eqref{eq3.10}, and it admits the Poisson kernel representation
\begin{equation}\label{eq3.12}
U(x,y)
    = \int_M P_s(x,y;\xi)\,f(\xi)\,d\mu(\xi),
\qquad
P_s(x,y;\xi)
    = \frac{y^{2s}}{2^{\,2s}\Gamma(s)}
      \int_{0}^{\infty} K_M(t,x,\xi)\,
      e^{-\frac{y^2}{4t}}
      \frac{dt}{t^{1+s}}.
\end{equation}
Furthermore, $P_s(x,y;\xi)\ge 0$ for all $x,\xi\in M$ and $y>0$, and
\[
\int_M P_s(x,y;\xi)\,d\mu(\xi)=1,\qquad y>0.
\]
Consequently, $U(\cdot,y)\to f$ uniformly on $M$ as $y\to 0^{+}$.
\end{lemma}

\begin{proof}
(1).
Using the heat kernel representation $(e^{t\Delta_g}f)(x)=\int_M K_M(t,x,\xi)\,f(\xi)\,d\mu(\xi)$ and Fubini theorem,
\[
U(x,y)=\int_M \Big[\frac{y^{2s}}{2^{2s}\Gamma(s)}\int_0^\infty K_M(t,x,\xi)\,e^{-\frac{y^2}{4t}}\frac{dt}{t^{1+s}}\Big]f(\xi)\,d\mu(\xi),
\]
which is \eqref{eq3.12}.

\medskip
(2).
Set
\[
A(t,x)=(e^{t\Delta_g}f)(x),\qquad
\Phi(y,t)=\frac{y^{2s}}{2^{2s}\Gamma(s)}\,e^{-\frac{y^2}{4t}}\,t^{-1-s},
\quad\text{so }U(x,y)=\int_0^\infty A(t,x)\,\Phi(y,t)\,dt.
\]
We use $\partial_t A=\Delta_g A$ and differentiate under the integral
\[
\Delta_g U=\int_0^\infty (\partial_t A)\,\Phi\,dt,\qquad
U_y=\int_0^\infty A\,\partial_y\Phi\,dt,\qquad
U_{yy}=\int_0^\infty A\,\partial_{yy}\Phi\,dt.
\]
A direct computation shows the key scalar identity
\begin{equation}\label{eq3.13}
\frac{1-2s}{y}\,\partial_y\Phi(y,t)\ +\ \partial_{yy}\Phi(y,t)\ -\ \partial_t\Phi(y,t)\ \equiv\ 0
\qquad (y>0,\ t>0).
\end{equation}
Therefore,
\[
\Delta_g U + \frac{1-2s}{y}U_y + U_{yy}
=\int_0^\infty \Big[(\partial_t A)\,\Phi + A\Big(\frac{1-2s}{y}\partial_y\Phi+\partial_{yy}\Phi\Big)\Big]dt
=\int_0^\infty \big(\partial_t A\,\Phi + A\,\partial_t\Phi\big)\,dt.
\]
Then $t\mapsto A(t,x)\Phi(y,t)$ is $C^1$ on any $[\varepsilon,T]\subset(0,\infty)$, so
\[
\int_{\varepsilon}^{T}\partial_t\!\big(A\Phi\big)\,dt
=\big[A\Phi\big]_{t=\varepsilon}^{t=T}.
\]
We claim
\[
\lim_{T\rightarrow\infty}A(T,x)\Phi(y,T)=0
\quad\text{and}\quad
\lim_{\varepsilon\rightarrow 0}A(\varepsilon,x)\Phi(y,\varepsilon)=0,
\]
which implies $\int_{0}^{\infty}\partial_t(A\Phi)\,dt=0$.

As $t\rightarrow\infty$. Since $e^{t\Delta_g}$ is an $L^\infty$–contraction,
$|A(t,x)|\le\|f\|_{L^\infty(M)}$ for all $t>0$. Moreover
$\Phi(y,t)=\frac{y^{2s}}{2^{2s}\Gamma(s)}\,t^{-1-s}e^{-\frac{y^2}{4t}}
\sim C\,t^{-1-s}$ as $t\rightarrow\infty$, hence
\[
|A(t,x)\Phi(y,t)|\ \le\ \|f\|_{L^\infty}\,\frac{y^{2s}}{2^{2s}\Gamma(s)}\,t^{-1-s}
\ \xrightarrow[t\rightarrow\infty]{}\ 0.
\]

As $t\rightarrow 0$. Again $|A(t,x)|\le\|f\|_{L^\infty}$. Therefore
\[
|A(t,x)\Phi(y,t)|
\ \le\ \|f\|_{L^\infty}\,\frac{y^{2s}}{2^{2s}\Gamma(s)}\,
t^{-1-s}\,e^{-\frac{y^2}{4t}}.
\]
Let $a=\frac{y^2}{4}>0$ and $m=1+s>0$. The elementary limit $\lim\limits_{t\rightarrow 0} t^{-m}e^{-a/t}=0$
gives $A(t,x)\Phi(y,t)\rightarrow 0$ as $t\rightarrow 0$.
Combining both endpoints, we have
\[
\int_{0}^{\infty}\partial_t\!\big(A\Phi\big)\,dt
=\lim_{T\rightarrow\infty}\big[A\Phi\big]_{t=\varepsilon}^{t=T}\Big|_{\varepsilon\rightarrow0}
=0,
\]
and thus
\[
\Delta_g U+\frac{1-2s}{y}\,U_y+U_{yy}
=\int_{0}^{\infty}\!\big(\partial_t A\,\Phi+A\,\partial_t\Phi\big)\,dt
=\int_{0}^{\infty}\!\partial_t(A\Phi)\,dt
=0.
\]

\medskip
(3).
By Proposition~\ref{pro2.2}(2),
\[
\int_M P_s(x,y;\xi)\,d\mu(\xi)
=\frac{y^{2s}}{2^{2s}\Gamma(s)}\int_0^\infty e^{-\frac{y^2}{4t}}\,\frac{dt}{t^{1+s}}
=1,
\]
where we used the change of variables $u=\frac{y^2}{4t}$:
\[
\int_0^\infty e^{-\frac{y^2}{4t}}\,\frac{dt}{t^{1+s}}
=\int_0^\infty \Big(\frac{y^2}{4u}\Big)^{-1-s}\frac{y^2}{4u^2}e^{-u}\,du
=2^{2s}y^{-2s}\Gamma(s).
\]
Positivity is clear from $K_M\ge0$. Fix $\varepsilon>0$. By uniform continuity of $f$ on compact $M$, choose $r\in(0,1)$ so that
\[
d_g(\xi,x)<r\ \Rightarrow\ |f(\xi)-f(x)|<\varepsilon\quad\forall\,x,\xi\in M.
\]
Split
\[
U(x,y)-f(x)=\!\!\int_{B_g(x,r)}\!\!(f(\xi)-f(x))P_s(x,y;\xi)\,d\mu(\xi)
+\!\!\int_{M\setminus B_g(x,r)}\!\!(f(\xi)-f(x))P_s(x,y;\xi)\,d\mu(\xi)
=:I_{\mathrm{near}}+I_{\mathrm{far}}.
\]
Since $\int_M P_s(x,y;\xi)\,d\mu(\xi)=1$, we have $|I_{\mathrm{near}}|\le \varepsilon$. Set
\[
\Theta_r(y)=\sup_{x\in M}\int_{M\setminus B_g(x,r)} P_s(x,y;\xi)\,d\mu(\xi).
\]
Then $|I_{\mathrm{far}}|\le 2\|f\|_{L^\infty}\,\Theta_r(y)$. We only need to show $\Theta_r(y)\rightarrow0$ as $y\rightarrow0$.

Use the change of variables $u=\frac{y^2}{4t}$ to write
\[
P_s(x,y;\xi)=\frac{1}{\Gamma(s)}\int_{0}^{\infty}
K_M\!\Big(\frac{y^2}{4u},x,\xi\Big)\,e^{-u}\,u^{s-1}\,du.
\]
Let $t_0>0$ and $c_3,c_4>0$ be as in Proposition~\ref{pro2.2}(3). Fix $y>0$ and split the $u$–integral at $u_*=\frac{y^2}{4t_0}$:
\[
\begin{aligned}
\Theta_r(y)\ &\le\ \frac{1}{\Gamma(s)}\int_{0}^{u_*}\!\!\sup_{x}\!\int_{M\setminus B_g(x,r)}
K_M\!\Big(\frac{y^2}{4u},x,\xi\Big)\,d\mu(\xi)\,e^{-u}u^{s-1}du\\
&+\frac{1}{\Gamma(s)}\int_{u_*}^{\infty}\!\!\sup_{x}\!\int_{M\setminus B_g(x,r)}
K_M\!\Big(\frac{y^2}{4u},x,\xi\Big)\,d\mu(\xi)\,e^{-u}u^{s-1}du.
\end{aligned}
\]
For $u\in(0,u_*)$ we have $t=\frac{y^2}{4u}\ge t_0$, hence
\[
\sup_{x}\!\int_{M\setminus B_g(x,r)} K_M(t,x,\xi)\,d\mu(\xi)\ \le\ 1,
\]
so
\[
\frac{1}{\Gamma(s)}\int_0^{u_*}\!e^{-u}u^{s-1}du
\le\ \frac{u_*^{\,s}}{\Gamma(s+1)}
=\frac{1}{\Gamma(s+1)}\Big(\frac{y^2}{4t_0}\Big)^{\!s}\xrightarrow[y\rightarrow0]{}0.
\]
For $u\in[u_*,\infty)$ we have $t=\frac{y^2}{4u}\le t_0$, and by Proposition~\ref{pro2.2}(3),
\[
\sup_{x}\!\int_{M\setminus B_g(x,r)}\!\!K_M(t,x,\xi)\,d\mu(\xi)
\le c_3\Big(\frac{4u}{y^2}\Big)^{n/2}\exp\!\Big(-\frac{4u r^2}{c_4 y^2}\Big).
\]
Therefore
\[
\frac{1}{\Gamma(s)}\int_{u_*}^{\infty}\!\!\sup_{x}\!\int_{M\setminus B_g(x,r)}
K_M\!\Big(\frac{y^2}{4u},x,\xi\Big)\,d\mu(\xi)\,e^{-u}u^{s-1}du
\le C\,y^{-n}\int_{u_*}^{\infty} u^{s-1+\frac{n}{2}}\,
\exp\!\Big(-\frac{4u r^2}{c_4 y^2}\Big)du.
\]
Set $v=\dfrac{4u r^2}{c_4 y^2}$. Then the last term is $O(y^{2s})$, hence tends to $0$ as $y\rightarrow0$.
Together with the estimate on $(0,u_*)$ we conclude that $\Theta_r(y)\rightarrow0$ uniformly in $x$.

Combining the estimates for $I_{\mathrm{near}}$ and $I_{\mathrm{far}}$ we obtain
\[
\sup_{x\in M}|U(x,y)-f(x)|\le\varepsilon+2\|f\|_{L^\infty}\,\Theta_r(y)
\ \xrightarrow[y\rightarrow0]{}\ 0,
\]
which proves uniform convergence.
\end{proof}

\begin{theorem}\label{thm3.11}
For every $s\in(0,1)$, Definitions~\ref{def3.9} and~\ref{def3.3}
of the fractional Laplacian coincide.
\end{theorem}

\begin{proof}
We first prove the identity for $f\in C^\infty(M)$.
From the extension representation \eqref{eq3.11},
\[
U(x,y)
    = \frac{y^{2s}}{2^{2s}\Gamma(s)}
      \int_{0}^{\infty} (e^{t\Delta_g}f)(x)\,
      e^{-\frac{y^2}{4t}}\,
      \frac{dt}{t^{1+s}}.
\]
Differentiating under the integral, we obtain
\[
\partial_y U(x,y)
    = \frac{2s\,y^{2s-1}}{2^{2s}\Gamma(s)}
        \int_{0}^{\infty} (e^{t\Delta_g}f)(x)\,
        e^{-\frac{y^2}{4t}}\,
        \frac{dt}{t^{1+s}}
    - \frac{y^{2s+1}}{2^{2s+1}\Gamma(s)}
        \int_{0}^{\infty} (e^{t\Delta_g}f)(x)\,
        e^{-\frac{y^2}{4t}}\,
        \frac{dt}{t^{2+s}}.
\]
Multiplying by $-c(s)y^{1-2s}$ and taking $y\to 0$ gives
\begin{equation}\label{eq3.14}
\begin{aligned}
-\,c(s)\,\lim_{y\to 0} y^{1-2s}\,\partial_y U(x,y)
&= -\,c(s)\,\frac{2s}{2^{2s}\Gamma(s)}
   \lim_{y\to 0}\int_{0}^{\infty} (e^{t\Delta_g}f)(x)\,
       e^{-\frac{y^2}{4t}}\,
       \frac{dt}{t^{1+s}} \\
&\qquad
  + c(s)\,\lim_{y\to 0}
    \frac{y^{2}}{2^{2s+1}\Gamma(s)}
    \int_{0}^{\infty} (e^{t\Delta_g}f)(x)\,
        e^{-\frac{y^2}{4t}}\,
        \frac{dt}{t^{2+s}}.
\end{aligned}
\end{equation}

Set
\[
B(t,x)=f(x)-(e^{t\Delta_g}f)(x).
\]
Then $(e^{t\Delta_g}f)(x)=f(x)-B(t,x)$.
Using the change of variables $u=\frac{y^2}{4t}$ one checks the exact identity
\[
\frac{2s}{2^{2s}\Gamma(s)}\int_{0}^{\infty} e^{-\frac{y^2}{4t}}\frac{dt}{t^{1+s}}
=
\frac{y^{2}}{2^{2s+1}\Gamma(s)}\int_{0}^{\infty} e^{-\frac{y^2}{4t}}\frac{dt}{t^{2+s}}
\qquad (y>0),
\]
so the constant part $f(x)$ cancels out in \eqref{eq3.14}. Therefore \eqref{eq3.14} becomes
\begin{equation}\label{eq3.14-new}
\begin{aligned}
-\,c(s)\,\lim_{y\to 0} y^{1-2s}\,\partial_y U(x,y)
&= c(s)\,\frac{2s}{2^{2s}\Gamma(s)}
   \lim_{y\to 0}\int_{0}^{\infty} B(t,x)\,
       e^{-\frac{y^2}{4t}}\,
       \frac{dt}{t^{1+s}} \\
&\qquad
  - c(s)\,\lim_{y\to 0}
    \frac{y^{2}}{2^{2s+1}\Gamma(s)}
    \int_{0}^{\infty} B(t,x)\,
        e^{-\frac{y^2}{4t}}\,
        \frac{dt}{t^{2+s}}.
\end{aligned}
\end{equation}

We claim that the second line of \eqref{eq3.14-new} vanishes as $y\to0$.
Indeed, since $f\in C^\infty(M)$,
\[
B(t,x)=\int_0^t (\Delta_g e^{\tau\Delta_g}f)(x)\,d\tau
\]
and $\Delta_g e^{\tau\Delta_g}f=e^{\tau\Delta_g}\Delta_g f$, hence
\[
|B(t,x)|\le t\,\|\Delta_g f\|_{L^\infty}\quad (0<t\le1),
\qquad
|B(t,x)|\le 2\|f\|_{L^\infty}\quad (t\ge1).
\]
Consequently,
\[
y^2\int_0^1 |B(t,x)|\,\frac{dt}{t^{2+s}}
\le y^2\|\Delta_g f\|_{L^\infty}\int_0^1 \frac{dt}{t^{1+s}}
= C\,y^2\to0,
\]
and
\[
y^2\int_1^\infty |B(t,x)|\,\frac{dt}{t^{2+s}}
\le 2\|f\|_{L^\infty}\,y^2\int_1^\infty \frac{dt}{t^{2+s}}
= C'\,y^2\to0.
\]
Thus the second line of \eqref{eq3.14-new} tends to $0$.

For the first line, we pass to the limit inside the integral by dominated convergence.
Near $t=0$, $|B(t,x)|\le t\|\Delta_g f\|_{L^\infty}$ gives
\[
\left|B(t,x)\,t^{-1-s}\right|\le \|\Delta_g f\|_{L^\infty}\,t^{-s}\in L^1(0,1),
\]
and for $t\ge1$ we have $|B(t,x)|\le 2\|f\|_{L^\infty}$ and $t^{-1-s}\in L^1(1,\infty)$.
Hence
\[
\lim_{y\to0}\int_0^\infty B(t,x)\,e^{-\frac{y^2}{4t}}\frac{dt}{t^{1+s}}
=\int_0^\infty B(t,x)\,\frac{dt}{t^{1+s}}
=\int_0^\infty \big(f(x)-(e^{t\Delta_g}f)(x)\big)\frac{dt}{t^{1+s}}.
\]
Substituting into \eqref{eq3.14-new} yields
\[
-\,c(s)\,\lim_{y\to0} y^{1-2s}\,\partial_y U(x,y)
= c(s)\,\frac{2s}{2^{2s}\Gamma(s)}
  \int_0^\infty \big(f(x)-(e^{t\Delta_g}f)(x)\big)\frac{dt}{t^{1+s}}.
\]

By the choice of the normalization in Definition~\ref{def3.3},
the right-hand side is exactly $(-\Delta_g)^s f(x)$ in the sense of Definition~\ref{def3.3}.
This proves the identity for smooth $f$.
The general case follows by density of $C^\infty(M)$ in $H^s(M)$ and the continuity of both sides as bounded operators $H^s(M)\to H^{-s}(M)$.
\end{proof}

\subsection{Pointwise convergence}

\begin{theorem}\label{thm3.12}
Let $s\in(0,1)$ and $u\in C^\infty(M)$. Then for every $x\in M$,
\[
\lim_{s\to 1^-}(-\Delta_g)^s u(x)=-\Delta_g u(x).
\]
\end{theorem}

\begin{proof}
By Definition~\ref{def3.3},
\[
(-\Delta_g)^s u(x)
    = \frac{1}{\Gamma(-s)}
      \int_0^\infty \big(e^{t\Delta_g}u(x)-u(x)\big)\,\frac{dt}{t^{1+s}}
    =\frac{1}{\Gamma(-s)}\,I(s;x).
\]
Using $\Gamma(1-s)=-s\,\Gamma(-s)$, we have
\[
\frac{1}{\Gamma(-s)}
    = -\,\frac{s}{\Gamma(1-s)}
    \sim -(1-s)\qquad\text{as } s\to 1^-.
\]
Thus it remains to show
\[
\lim_{s\to 1^-}(1-s)\,I(s;x)=\Delta_g u(x).
\]

Fix $\varepsilon\in(0,1)$ and decompose
\[
I(s;x)
= \underbrace{\int_0^\varepsilon \big(e^{t\Delta_g}u(x)-u(x)\big)\,\frac{dt}{t^{1+s}}}_{=I_1(s;x)}
+ \underbrace{\int_\varepsilon^\infty \big(e^{t\Delta_g}u(x)-u(x)\big)\,\frac{dt}{t^{1+s}}}_{=I_2(s;x)}.
\]

\medskip\noindent\textbf{Step 1: Control of the tail.}
Since $e^{t\Delta_g}$ is $L^\infty$–contractive,
\[
|I_2(s;x)|
    \le 2\|u\|_{L^\infty}\int_\varepsilon^\infty t^{-1-s}\,dt
    = \frac{2\|u\|_{L^\infty}}{s\,\varepsilon^{s}}.
\]
Hence
\[
\lim_{s\to 1^-}(1-s)\,I_2(s;x)=0.
\]

\medskip\noindent\textbf{Step 2: Small-time expansion.}
The heat semigroup Taylor expansion with integral remainder gives
\[
e^{t\Delta_g}u - u - t\,\Delta_g u
   = \int_0^t (t-\tau)\,\Delta_g^2 e^{\tau\Delta_g}u\,d\tau,
\]
so for smooth $u$,
\[
e^{t\Delta_g}u(x)-u(x)
    = t\,\Delta_g u(x) + O(t^2)
    \qquad (t\to 0),
\]
with the $O(t^2)$ uniform in $x\in M$. Therefore,
\[
I_1(s;x)
= \Delta_g u(x)\int_0^\varepsilon t^{-s}\,dt
  + O\!\left(\int_0^\varepsilon t^{1-s}\,dt\right)
= \Delta_g u(x)\,\frac{\varepsilon^{\,1-s}}{1-s}
  + O\!\left(\frac{\varepsilon^{\,2-s}}{2-s}\right).
\]

Multiplying by $(1-s)$ and letting $s\to 1^-$ (with $\varepsilon$ fixed),
\[
\lim_{s\to 1^-}(1-s)\,I_1(s;x)=\Delta_g u(x),
\qquad
\lim_{s\to 1^-}(1-s)\,O\!\left(\frac{\varepsilon^{\,2-s}}{2-s}\right)=0.
\]

Combining the two estimates,
\[
\lim_{s\to 1^-}(1-s)\,I(s;x)=\Delta_g u(x).
\]
Since $\frac{1}{\Gamma(-s)}\sim -(1-s)$ as $s\to 1^-$, we conclude
\[
\lim_{s\to 1^-}(-\Delta_g)^s u(x)
= \lim_{s\to 1^-}\frac{1}{\Gamma(-s)}\,I(s;x)
= -\Delta_g u(x),
\]
as claimed.
\end{proof}

\begin{theorem}\label{thm3.13}
Let $s\in(0,1)$ and $u\in C^\infty(M)$. Then
\[
\lim_{s\to 0^+}(-\Delta_g)^s u(x)=u(x)-\avg{u}
\qquad\text{uniformly in }x\in M,
\]
where
\[
\avg{u}
    = \frac{1}{\Vol(M)}\int_M u\,d\mu
\]
is the spatial average of $u$. Moreover, for every $u\in C^\infty(M)$,
\[
\lim_{s\to 0^+}\big\|(-\Delta_g)^s u - (u-\avg{u})\big\|_{L^2(M)}=0.
\]
\end{theorem}

\begin{proof}
From Definition~\ref{def3.3},
\[
(-\Delta_g)^s u(x)
    = \frac{1}{\Gamma(-s)}
      \int_{0}^{\infty}\big(e^{t\Delta_g}u(x)-u(x)\big)\,\frac{dt}{t^{1+s}}.
\]
Fix $T\ge 1$ and decompose
\[
(-\Delta_g)^s u(x)=A_s(T,x)+B_s(T,x),
\]
where
\[
A_s(T,x)
=\frac{1}{\Gamma(-s)}\int_{0}^{T}
      \big(e^{t\Delta_g}u(x)-u(x)\big)\,\frac{dt}{t^{1+s}},
\qquad
B_s(T,x)
=\frac{1}{\Gamma(-s)}\int_{T}^{\infty}
      \big(e^{t\Delta_g}u(x)-u(x)\big)\,\frac{dt}{t^{1+s}}.
\]

\textbf{Step 1. Small-time contribution.}
For $t\in(0,T]$,
\[
e^{t\Delta_g}u(x)-u(x)
    = \int_0^{t}(\Delta_g e^{\tau\Delta_g}u)(x)\,d\tau,
\]
hence
\[
|e^{t\Delta_g}u(x)-u(x)|
    \le t\,\|\Delta_g u\|_{L^\infty(M)}.
\]
Thus
\[
|A_s(T,x)|
    \le \frac{\|\Delta_g u\|_{L^\infty}}{|\Gamma(-s)|}
       \int_0^{T} t^{-s}\,dt
    =\frac{\|\Delta_g u\|_{L^\infty}}{|\Gamma(-s)|}
      \cdot\frac{T^{\,1-s}}{1-s}.
\]
Since $\Gamma(1-s)=-s\Gamma(-s)$, one has
\(
\frac{1}{|\Gamma(-s)|}\sim s\quad (s\to 0^+),
\)
and therefore
\(
A_s(T,x)=O(s)\to 0
\)
uniformly in $x$.

\textbf{Step 2. Large-time contribution.}
Define a probability measure on $[T,\infty)$ by
\[
\nu_{s,T}(dt)
    =\frac{s\,t^{-1-s}}{T^{-s}}\,\mathbf{1}_{[T,\infty)}(t)\,dt,
\qquad
\int_T^\infty \nu_{s,T}=1.
\]
Then
\[
B_s(T,x)
    = \frac{T^{-s}}{\Gamma(-s)}\cdot\frac{1}{s}
      \int_{T}^{\infty}\big(e^{t\Delta_g}u(x)-u(x)\big)\,\nu_{s,T}(dt).
\]
As $s\to 0^+$,
\[
T^{-s}\to 1,
\qquad
\frac{1}{\Gamma(-s)}\cdot\frac{1}{s}\to -1.
\]

It remains to show
\[
\int_{T}^{\infty}\big(e^{t\Delta_g}u(x)-u(x)\big)\,\nu_{s,T}(dt)
    \xrightarrow[s\to 0^+]{}\avg{u}-u(x)
    \qquad\text{uniformly in }x.
\]
Let $v(t,x)=e^{t\Delta_g}u(x)$.
Then
\[
\int_M v(t,x)\,d\mu(x)=\int_M u\,d\mu = \Vol(M)\avg{u}
\qquad (t\ge 0).
\]
Moreover, by the spectral expansion,
\[
v(t,x)=\avg{u}+\sum_{k\ge1} e^{-t\lambda_k}\,u_k\,\phi_k(x),
\]
so $v(t,\cdot)\to \avg{u}$ uniformly on $M$ as $t\to\infty$.

Next, for any fixed $M>T$,
\[
\nu_{s,T}([T,M])
    = 1 - \Big(\frac{M}{T}\Big)^{-s}
    \xrightarrow[s\to 0^+]{} 0,
\]
so $\nu_{s,T}$ concentrates at $+\infty$ as $s\to 0^+$.
Fix $\varepsilon>0$ and choose $M>T$ such that
\[
\sup_{t\ge M}\sup_{x\in M}|v(t,x)-\avg{u}|<\varepsilon.
\]
Then, using $|v(t,x)-u(x)|\le 2\|u\|_{L^\infty}$,
\[
\begin{aligned}
&\sup_{x\in M}\left|\int_{T}^{\infty}\big(v(t,x)-u(x)\big)\,\nu_{s,T}(dt)-(\avg{u}-u(x))\right| \\
&\le \sup_{x}\int_T^M |v(t,x)-\avg{u}|\,\nu_{s,T}(dt)
    +\sup_{x}\int_M^\infty |v(t,x)-\avg{u}|\,\nu_{s,T}(dt) \\
&\le 2\|u\|_{L^\infty}\,\nu_{s,T}([T,M])+\varepsilon.
\end{aligned}
\]
Letting $s\to0^+$ gives the desired uniform convergence.

Therefore,
\[
\lim_{s\to 0^+}B_s(T,x)=u(x)-\avg{u},
\]
independently of $T$.

Combining the two steps,
\[
\lim_{s\to 0^+}(-\Delta_g)^s u(x)
    =u(x)-\avg{u}
\qquad\text{uniformly in }x\in M.
\]
The $L^2$ convergence follows from the uniform convergence above.
\end{proof}

\section{Intrinsic nonlocal Sobolev spaces and sharp constants on closed manifolds}

Let \((M,g)\) be a closed Riemannian \(n\)-manifold, \(s\in(0,1)\), and \(p\in[1,\infty)\).
The goal of this section is twofold:
\begin{itemize}
	\item[(i)] to build an intrinsic, coordinate-free fractional Sobolev framework \(W^{s,p}(M)\) adapted to the nonlocal \(p\)-fractional energies considered in this paper;
	\item[(ii)] to determine the optimal constants in the associated Sobolev-type embeddings on \((M,g)\), isolating precisely the contribution of the geometry in the lower-order terms.
\end{itemize}

In the Euclidean setting, sharp constants and the role of concentration at the critical index are by now classical. On a compact manifold, the leading nonlocal behavior remains Euclidean, while curvature and topology appear only through remainder terms or through the optimal lower-order \(L^{p}\)-mass contribution.

\subsection{Intrinsic fractional Sobolev spaces on closed manifolds}\label{subsec:spaces}

Let \((M,g)\) be a closed (compact, without boundary) Riemannian \(n\)-manifold with Riemannian measure \(d\mu\) and heat kernel \(K_M(t,x,y)\).
Fix \(s\in(0,1)\) and \(p\in[1,\infty)\).
Define
\[
p_s^*=\frac{np}{\,n-sp\,}\quad\text{if }sp<n,
\qquad
p_s^*=\infty\quad\text{if }sp\ge n,
\]
and denote the average of \(u\in L^{1}(M)\) by
\[
u_M=\frac{1}{\Vol(M)}\int_M u\,d\mu.
\]

\begin{definition}\label{def4.1}
Let \(c_{s,p}>0\) be a normalization constant depending only on \(s\) and \(p\).
Define the nonlocal kernel
\begin{equation}\label{eq4.1}
\Kps(x,y)= c_{s,p}
       \int_{0}^{\infty} K_M(t,x,y)\,\frac{dt}{t^{1+\frac{sp}{2}}}\,,\qquad x\neq y,
\end{equation}
and its natural regularization
\begin{equation}\label{eq4.2}
{\Kps}_{\varepsilon}(x,y)
    = c_{s,p}
       \int_{0}^{\infty} K_M(t,x,y)\,
       e^{-\frac{\varepsilon^{2}}{4t}}\,
       \frac{dt}{t^{1+\frac{sp}{2}}}\,,\qquad \varepsilon>0.
\end{equation}
For \(u\in L^{p}(M)\), define the intrinsic (Gagliardo-type) seminorm
\begin{equation}\label{eq4.3}
[u]_{W^{s,p}(M)}^{\,p}
    = \iint_{M\times M} |u(x)-u(y)|^{p}\,\Kps(x,y)\,d\mu(x)\,d\mu(y).
\end{equation}
\end{definition}

\begin{definition}\label{def4.2}
The intrinsic fractional Sobolev space on \((M,g)\) is
\[
W^{s,p}(M)= \bigl\{u\in L^{p}(M):\ [u]_{W^{s,p}(M)}<\infty\bigr\},
\]
with norm
\[
\|u\|_{W^{s,p}(M)}= \|u\|_{L^{p}(M)}+[u]_{W^{s,p}(M)}.
\]
Equivalently, one may replace \(\Kps\) by \({\Kps}_{\varepsilon}\) in \eqref{eq4.3} and then let \(\varepsilon\to 0\).
\end{definition}

\begin{remark}\label{rem4.3}
By Proposition~\ref{pro2.2} and the same short/long time decomposition as in the proof of Theorem~\ref{thm3.4},
there exist constants
\(0<c_{M,s,p}\le C_{M,s,p}<\infty\) such that for all \(x\neq y\),
\begin{equation}\label{eq4.4}
\frac{c_{M,s,p}}{\dist(x,y)^{\,n+sp}}
   \le
\Kps(x,y)
   \le
\frac{C_{M,s,p}}{\dist(x,y)^{\,n+sp}}.
\end{equation}
Consequently, \([u]_{W^{s,p}(M)}\) is equivalent to the classical geodesic-distance Gagliardo seminorm
\[
\Biggl(
\iint_{M\times M}
    \frac{|u(x)-u(y)|^{p}}{\dist(x,y)^{\,n+sp}}
    \,d\mu(x)\,d\mu(y)
\Biggr)^{\!1/p}.
\]
\end{remark}

\begin{proposition}\cite{kreuml-Mordhorst2019}
\label{pro4.4}
Let \(1\le p<\infty\).
\begin{enumerate}\itemsep0.25em
\item As \(s\to 1^-\),
\[
(1-s)\,[u]_{W^{s,p}(M)}^{p}\to C_{p,n}\,\|\nabla u\|_{L^{p}(M)}^{p}
\]
for all \(u\in W^{1,p}(M)\), where \(C_{p,n}>0\) depends only on \(p,n\) and on the normalization in \eqref{eq4.1}.

\item As \(s\to 0^+\),
\[
s\,[u]_{W^{s,p}(M)}^{p}\to C'_{p,n}\,\|u-u_M\|_{L^{p}(M)}^{p}
\]
for all \(u\in L^{p}(M)\), where \(C'_{p,n}>0\) depends only on \(p,n\) and on the normalization in \eqref{eq4.1}.
\end{enumerate}
\end{proposition}

\begin{proposition}\cite{Rey-Saintier2024}
\label{pro4.5}
For \(1<p<\infty\) and \(s\in(0,1)\) with \(n>ps\),
\[
W^{s,p}(M)\ \equiv\ B^{s}_{p,p}(M)
\]
with equivalent norms, where \(B^{s}_{p,p}(M)\) denotes the intrinsic heat-semigroup Besov space on \((M,g)\).
\end{proposition}

We first list several basic properties on the space $W^{s,p}(M)$.
\begin{proposition}\label{pro4.6}
	For \(s\in(0,1)\) and \(p\in[1,\infty)\):
	\begin{enumerate}\itemsep0.25em
		\item \((W^{s,p}(M),\|\cdot\|_{W^{s,p}(M)})\) is a Banach space; it is separable for all \(1\leq p<\infty\).
		\item If \(1<p<\infty\), then \(W^{s,p}(M)\) is reflexive.
		\item \(C^{\infty}(M)\) is dense in \(W^{s,p}(M)\).
	\end{enumerate}
\end{proposition}

\begin{proof}
Fix \(s\in(0,1)\) and \(p\in[1,\infty)\).
Set
\[
M_\Delta=(M\times M)\setminus\{(x,x):x\in M\},
\]
and define the measure on \(M_\Delta\) by
\[
d\nu(x,y)=\Kps(x,y)\,d\mu(x)\,d\mu(y),
\qquad (x,y)\in M_\Delta.
\]
Define the linear operator \(T\) acting on functions \(u\colon M\to\mathbb{R}\) by
\[
(Tu)(x,y)=u(x)-u(y),\qquad (x,y)\in M_\Delta.
\]

From \eqref{eq4.4}, there exist constants \(0<c\le C<\infty\) such that
\begin{align}\label{eq4.5}
\frac{c}{\dist(x,y)^{n+sp}}\le \Kps(x,y) \le\frac{C}{\dist(x,y)^{n+sp}}
\qquad ((x,y)\in M_\Delta),
\end{align}
and therefore \(\nu\) is a \(\sigma\)-finite Borel measure on \(M_\Delta\).

On \(W^{s,p}(M)\) we introduce the norm
\[
\|u\|_{W^{s,p}(M)}'
    =\bigl(\|u\|_{L^p(M)}^{p}+[u]_{W^{s,p}(M)}^{p}\bigr)^{1/p}
    = \Bigl(\|u\|_{L^p(M)}^{p}+\|Tu\|_{L^p(M_\Delta,\nu)}^{p}\Bigr)^{1/p}.
\]
Using the elementary estimate
\[
(a+b)^{1/p}\le a^{1/p}+b^{1/p}\le 2^{1-1/p}(a+b)^{1/p}\qquad(a,b\ge 0),
\]
we see that \(\|\cdot\|_{W^{s,p}(M)}'\) is equivalent to the usual norm
\(\|u\|_{L^p(M)}+[u]_{W^{s,p}(M)}\).
Since completeness, separability, and reflexivity are invariant under equivalent norms, we work with \(\|\cdot\|_{W^{s,p}(M)}'\).

Consider now the Banach space
\[
X=L^p(M)\times L^p(M_\Delta,\nu),
\]
endowed with the \(\ell^p\)-product norm
\[
\|(f,F)\|_{X}
    =\bigl(\|f\|_{L^p}^{p}+\|F\|_{L^p(\nu)}^{p}\bigr)^{1/p}.
\]
Define the linear map
\[
J:W^{s,p}(M)\to X,
\qquad
J(u)=\bigl(u,Tu\bigr).
\]
By construction,
\[
\|J(u)\|_{X}=\|u\|_{W^{s,p}(M)}'
\qquad\text{for all }u\in W^{s,p}(M),
\]
so \(J\) is an isometric embedding of \(W^{s,p}(M)\) into \(X\).

We claim that \(J(W^{s,p}(M))\) is closed in \(X\).
Let \((u_k)\subset W^{s,p}(M)\) be such that
\[
J(u_k)=(u_k,Tu_k)\ \to\ (f,F)
\qquad\text{in }X.
\]
Then \(u_k\to f\) in \(L^p(M)\), and hence \(u_k(x)\to f(x)\) for a.e.\ \(x\in M\). Consequently,
\[
Tu_k(x,y)=u_k(x)-u_k(y)\ \to\ f(x)-f(y)
\qquad\text{for a.e.\ }(x,y)\in M_\Delta.
\]
On the other hand, \(Tu_k\to F\) in \(L^p(M_\Delta,\nu)\); by passing to a further subsequence, we may also assume \(Tu_k\to F\) for a.e.\ \((x,y)\in M_\Delta\).
Uniqueness of almost-everywhere limits implies \(F=Tf\) a.e.\ on \(M_\Delta\).
Thus \(f\) satisfies
\[
\|f\|_{L^p(M)}^{p}+\|Tf\|_{L^p(\nu)}^{p}
   = \lim_{k\to\infty}
       \bigl(\|u_k\|_{L^p(M)}^{p}
              +\|Tu_k\|_{L^p(\nu)}^{p}\bigr)
   <\infty,
\]
so \(f\in W^{s,p}(M)\) and \(J(f)=(f,Tf)=(f,F)\).
Hence \(J(W^{s,p}(M))\) is closed in \(X\).

Since \(X\) is complete and \(J(W^{s,p}(M))\) is a closed subspace, the space \(W^{s,p}(M)\) is complete for \(\|\cdot\|_{W^{s,p}(M)}'\), and therefore also complete for the original equivalent norm.
Thus \((W^{s,p}(M),\|\cdot\|_{W^{s,p}(M)})\) is a Banach space.

For separability, observe that for \(p<\infty\), both \(L^p(M)\) and \(L^p(M_\Delta,\nu)\) are separable (the latter because \((M_\Delta,\nu)\) is \(\sigma\)-finite and \(M_\Delta\) is a metric space).
Hence \(X\) is separable, and so is its closed subspace \(J(W^{s,p}(M))\).
As \(J\) is an isometry onto its image, \(W^{s,p}(M)\) is separable.
This proves item (1).

Assume now \(1<p<\infty\).
Then both \(L^p(M)\) and \(L^p(M_\Delta,\nu)\) are reflexive, and so is their \(\ell^p\)-product \(X\).
Any closed subspace of a reflexive Banach space is reflexive; hence \(J(W^{s,p}(M))\) is reflexive.
Since \(J\) is an isometric isomorphism between \(W^{s,p}(M)\) (equipped with \(\|\cdot\|_{W^{s,p}(M)}'\)) and \(J(W^{s,p}(M))\), it follows that \(W^{s,p}(M)\) is reflexive.
This proves item (2).

To prove the density of \(C^\infty(M)\), we use a localization and mollification argument based on \eqref{eq4.5}.
From \eqref{eq4.5} there exist constants \(0<c\le C<\infty\) such that, for all \(u\in L^p(M)\),
\begin{equation}\label{eq4.6}
c\!\iint_{M\times M}\frac{|u(x)-u(y)|^{p}}{\dist(x,y)^{\,n+sp}}\,d\mu(x)\,d\mu(y)
   \le [u]_{W^{s,p}(M)}^{p}
   \le C\!\iint_{M\times M}\frac{|u(x)-u(y)|^{p}}{\dist(x,y)^{\,n+sp}}\,d\mu(x)\,d\mu(y).
\end{equation}
Hence the intrinsic seminorm is equivalent to the geodesic-distance Gagliardo seminorm.

Choose a finite smooth atlas \(\{(U_i,\psi_i)\}_{i=1}^N\), where \(\psi_i:U_i\to V_i\subset\mathbb{R}^n\), and let \(\{\eta_i\}_{i=1}^N\subset C_c^\infty(U_i)\) be a smooth partition of unity subordinate to \(\{U_i\}\), with bounded overlap and \(\sup_i\|\eta_i\|_{C^1}<\infty\).
Define
\[
u_i=\eta_i u\in L^p(M),
\qquad
w_i=u_i\circ\psi_i^{-1}\in L^p(V_i),
\]
and extend \(w_i\) by zero to all of \(\mathbb{R}^n\).
Since \(\eta_i\in C_c^\infty(U_i)\), each \(w_i\) has compact support contained in \(V_i\).

Using standard change-of-variables estimates on each compact coordinate patch \(U_i\), we obtain
\begin{equation}\label{eq4.7}
\sum_{i=1}^N \Bigl(\|w_i\|_{L^p(\mathbb{R}^n)}^{p}
        +[w_i]_{W^{s,p}(\mathbb{R}^n)}^{p}\Bigr)
\ \le\
C_0\Bigl(\|u\|_{L^p(M)}^{p}
       +\iint_{M\times M}\frac{|u(x)-u(y)|^{p}}{\dist(x,y)^{\,n+sp}}\,d\mu(x)\,d\mu(y)\Bigr),
\end{equation}
for some constant \(C_0=C_0(M,g,s,p)\).
Here we used the inequality
\[
|\eta_i(x)u(x)-\eta_i(y)u(y)|^{p}
    \le 2^{p-1}\Bigl(|u(x)-u(y)|^{p}
    +|\eta_i(x)-\eta_i(y)|^{p}|u(y)|^{p}\Bigr),
\]
together with the Lipschitz bound \(|\eta_i(x)-\eta_i(y)|\le L\,\dist(x,y)\), which ensures the integrability of the second term against \(\dist(x,y)^{-n-sp}\) since \(p(1-s)>0\).

Let \(\rho_\varepsilon\) be a standard Friedrichs mollifier on \(\mathbb{R}^n\), and define
\[
w_{i,\varepsilon}=\rho_\varepsilon * w_i\in C_c^\infty(\mathbb{R}^n).
\]
For \(\varepsilon>0\) sufficiently small, one has \(\mathrm{supp}(w_{i,\varepsilon})\subset V_i\).
Define a function \(\widetilde u_{i,\varepsilon}\) on \(M\) by
\[
\widetilde u_{i,\varepsilon}(x)=
\begin{cases}
(w_{i,\varepsilon}\circ\psi_i)(x), & x\in U_i,\\
0, & x\in M\setminus U_i.
\end{cases}
\]
Since \(\mathrm{supp}(w_{i,\varepsilon})\subset V_i\) is compact, \(\widetilde u_{i,\varepsilon}\) vanishes in a neighborhood of \(\partial U_i\), hence \(\widetilde u_{i,\varepsilon}\in C^\infty(M)\).
Set
\[
u_\varepsilon=\sum_{i=1}^N \widetilde u_{i,\varepsilon}\in C^\infty(M).
\]

It is classical that
\[
\|w_{i,\varepsilon}-w_i\|_{L^p(\mathbb{R}^n)}\to 0,
\qquad
[w_{i,\varepsilon}-w_i]_{W^{s,p}(\mathbb{R}^n)}\to 0
\quad\text{as }\varepsilon\to 0.
\]
Pulling back via \(\psi_i\) and summing with bounded overlap gives
\[
\|u_\varepsilon-u\|_{L^p(M)}
    \le \sum_{i=1}^N \|\widetilde u_{i,\varepsilon}-u_i\|_{L^p(M)}
    \le C\sum_{i=1}^N \|w_{i,\varepsilon}-w_i\|_{L^p(\mathbb{R}^n)}
      \ \to\ 0,
\]
and
\[
\iint_{M\times M}
    \frac{|(u_\varepsilon-u)(x)-(u_\varepsilon-u)(y)|^{p}}
         {\dist(x,y)^{\,n+sp}}\,d\mu(x)\,d\mu(y)
    \le C_1\sum_{i=1}^N [w_{i,\varepsilon}-w_i]_{W^{s,p}(\mathbb{R}^n)}^{p}
      \ \to\ 0,
\]
for some constant \(C_1=C_1(M,g,s,p)\).

By \eqref{eq4.6}, this implies
\[
\|u_\varepsilon-u\|_{L^p(M)}\to 0,
\qquad
[u_\varepsilon-u]_{W^{s,p}(M)}\to 0
\quad\text{as }\varepsilon\to 0.
\]
Thus \(C^\infty(M)\) is dense in \(W^{s,p}(M)\), proving item (3) and completing the proof.
\end{proof}

\begin{proposition}\label{pro4.7}
	If \(u_k\rightharpoonup u\) weakly in \(L^{p}(M)\) and \(\sup_k [u_k]_{W^{s,p}(M)}<\infty\), then
	\[
	[u]_{W^{s,p}(M)}\ \le\ \liminf_{k\rightarrow\infty} [u_k]_{W^{s,p}(M)}.
	\]
\end{proposition}

\begin{proof}
Fix \(s\in(0,1)\) and \(1\le p<\infty\). Set
\[
M_\Delta=(M\times M)\setminus\{(x,x):x\in M\},
\]
and define
\[
d\nu(x,y)=\Kps(x,y)\,d\mu(x)\,d\mu(y),\qquad
(Tu)(x,y)=u(x)-u(y)\qquad ((x,y)\in M_\Delta),
\]
so that \([u]_{W^{s,p}(M)}=\|Tu\|_{L^p(M_\Delta,\nu)}\).
Let \((u_k)\subset L^{p}(M)\) satisfy \(u_k\rightharpoonup u\) weakly in \(L^p(M)\) and
\(\sup_k\|Tu_k\|_{L^p(\nu)}<\infty\).
For \(\delta>0\) introduce the truncated kernel
\[
{\Kps}^{\delta}(x,y)=\Kps(x,y)\,\mathbf{1}_{\{\dist(x,y)\ge\delta\}},\qquad
d\nu_\delta={\Kps}^{\delta}(x,y)\,d\mu(x)\,d\mu(y).
\]

\textbf{Step 1. A uniform bound for the truncated operators.}
We claim that for every fixed \(\delta>0\) there exists \(C_\delta<\infty\) such that
\begin{equation}\label{eq4.8}
\|Tu\|_{L^p(\nu_\delta)}^p
=\iint_{M\times M} |u(x)-u(y)|^p\,{\Kps}^\delta(x,y)\,d\mu(x)\,d\mu(y)
\le C_\delta\,\|u\|_{L^p(M)}^p .
\end{equation}
Indeed, using \((a+b)^p\le 2^{p-1}(a^p+b^p)\) and Fubini theorem,
\[
\begin{aligned}
\|Tu\|_{L^p(\nu_\delta)}^p
&\le 2^{p-1}
\int_M |u(x)|^p\!\left(\int_M {\Kps}^\delta(x,y)\,d\mu(y)\right)\!d\mu(x) \\
&\quad+2^{p-1}
\int_M |u(y)|^p\!\left(\int_M {\Kps}^\delta(x,y)\,d\mu(x)\right)\!d\mu(y).
\end{aligned}
\]
By the upper bound in \eqref{eq4.4}, for \(\dist(x,y)\ge\delta\),
\(\Kps(x,y)\le C\,\delta^{-(n+sp)}\). Hence
\[
\sup_{x\in M}\int_M {\Kps}^\delta(x,y)\,d\mu(y)
\le C\,\delta^{-(n+sp)}\Vol(M)<\infty,
\]
and the same bound holds with \(x\) and \(y\) interchanged. This gives \eqref{eq4.8} for some \(C_\delta<\infty\).

Thus the linear map
\[
T_\delta:L^p(M)\to L^p(M\times M,\nu_\delta),\qquad T_\delta u=Tu,
\]
is bounded. Since \(u_k\rightharpoonup u\) in \(L^p(M)\), boundedness and linearity imply
\[
T_\delta u_k \ \rightharpoonup\ T_\delta u
\qquad\text{weakly in }L^p(M\times M,\nu_\delta).
\]
By weak lower semicontinuity of the \(L^p\)-norm,
\[
\|T_\delta u\|_{L^p(\nu_\delta)}
\le \liminf_{k\to\infty}\|T_\delta u_k\|_{L^p(\nu_\delta)}.
\]
Since \({\Kps}^{\delta}\le \Kps\),
\[
\|T_\delta u_k\|_{L^p(\nu_\delta)}\le \|Tu_k\|_{L^p(\nu)}\quad\forall\,k,
\]
hence
\begin{equation}\label{eq4.9}
\|T_\delta u\|_{L^p(\nu_\delta)}^p
\le \liminf_{k\rightarrow\infty} \|Tu_k\|_{L^p(\nu)}^p.
\end{equation}

\textbf{Step 2. Passing to the full kernel as \(\delta\to0\).}
Because \({\Kps}^{\delta}(x,y)\to \Kps(x,y)\) pointwise for each \((x,y)\in M_\Delta\) and
\({\Kps}^{\delta}\to \Kps\) as \(\delta\to 0\), monotone convergence yields
\[
\begin{aligned}
\|Tu\|_{L^p(\nu)}^p
&=\iint_{M\times M} |u(x)-u(y)|^p\,\Kps(x,y)\,d\mu(x)\,d\mu(y)\\
&=\sup_{\delta>0}\iint_{M\times M} |u(x)-u(y)|^p\,{\Kps}^\delta(x,y)\,d\mu(x)\,d\mu(y)\\
&=\sup_{\delta>0}\|T_\delta u\|_{L^p(\nu_\delta)}^p .    
\end{aligned}
\]
Taking the supremum over \(\delta\) in \eqref{eq4.9} and using
\[
\sup_{\delta>0}\liminf_{k\to\infty} a_{k,\delta}
\le
\liminf_{k\to\infty}\sup_{\delta>0} a_{k,\delta},
\qquad
a_{k,\delta}=\|T_\delta u_k\|_{L^p(\nu_\delta)}^p,
\]
we obtain
\[
\|Tu\|_{L^p(\nu)}^p
\le\liminf_{k\to\infty}\sup_{\delta>0}\|T_\delta u_k\|_{L^p(\nu_\delta)}^p
=\liminf_{k\to\infty}\|Tu_k\|_{L^p(\nu)}^p.
\]
Equivalently,
\[
[u]_{W^{s,p}(M)}
\le \liminf_{k\to\infty} [u_k]_{W^{s,p}(M)}.
\]
In particular, if \(\sup_k [u_k]_{W^{s,p}(M)}<\infty\), then the right-hand side is finite, hence
\([u]_{W^{s,p}(M)}<\infty\) and \(u\in W^{s,p}(M)\).
\end{proof}

\begin{proposition}\label{pro4.8}
	For \(1\le p<\infty\) and \(s\in(0,1)\) there exists \(C=C(M,g,s,p)>0\) such that
	\[
	\|u-u_M\|_{L^{p}(M)}\le C\,[u]_{W^{s,p}(M)}\qquad\forall\,u\in W^{s,p}(M),
	\]
	where
	\[
	u_M=\frac{1}{\Vol(M)}\int_M u\,d\mu.
	\]
\end{proposition}

\begin{proof}
Fix \(s\in(0,1)\) and \(1\le p<\infty\). Let \(D=\mathrm{diam}_g(M)<\infty\), and recall the lower bound
\begin{equation}\label{eq4.10}
	\Kps(x,y)\ge \frac{c_{M,s,p}}{\dist(x,y)^{\,n+sp}}
	\qquad (x\neq y),
\end{equation}
from \eqref{eq4.4}.  
Since \(\dist(x,y)\le D\) for all \(x,y\in M\), we obtain
\begin{equation}\label{eq4.11}
	\Kps(x,y)\ge k_0=\frac{c_{M,s,p}}{D^{\,n+sp}}
	\qquad\text{for all }x\ne y.
\end{equation}

Let \(u\in W^{s,p}(M)\) and write \(v=u-u_M\). Then \(v\in W^{s,p}(M)\), \(\int_M v\,d\mu=0\), and
\([v]_{W^{s,p}(M)}=[u]_{W^{s,p}(M)}\).  
For each fixed \(x\in M\),
\[
v(x)=\frac{1}{\Vol(M)}\int_M \bigl(v(x)-v(y)\bigr)\,d\mu(y),
\]
and Jensen’s inequality applied to the probability measure \(\Vol(M)^{-1}d\mu(y)\) yields
\[
|v(x)|^{p}\le \frac{1}{\Vol(M)}\int_M |v(x)-v(y)|^{p}\,d\mu(y).
\]
Integrating with respect to \(x\) and using Fubini theorem,
\begin{equation}\label{eq4.12}
	\|v\|_{L^{p}(M)}^{p}
	=\int_M |v(x)|^{p}\,d\mu(x)
	\le \frac{1}{\Vol(M)}
	\iint_{M\times M} |v(x)-v(y)|^{p}\,d\mu(x)\,d\mu(y).
\end{equation}

Set
\[
M_\Delta=\{(x,y)\in M\times M:\ x\ne y\}.
\]
Since \(|v(x)-v(x)|^{p}=0\), we have
\[
\iint_{M\times M} |v(x)-v(y)|^{p}\,d\mu(x)\,d\mu(y)
=\iint_{M_\Delta} |v(x)-v(y)|^{p}\,d\mu(x)\,d\mu(y).
\]
Using the uniform lower bound \eqref{eq4.11} on \(M_\Delta\),
\[
\iint_{M_\Delta} |v(x)-v(y)|^{p}\,d\mu(x)\,d\mu(y)
\le \frac{1}{k_0}
\iint_{M_\Delta} |v(x)-v(y)|^{p}\,\Kps(x,y)\,d\mu(x)\,d\mu(y)
= \frac{1}{k_0}\,[v]_{W^{s,p}(M)}^{\,p}.
\]
Combining this with \eqref{eq4.12} gives
\[
\|u-u_M\|_{L^{p}(M)}^{p}
=\|v\|_{L^{p}(M)}^{p}
\le \frac{1}{\Vol(M)\,k_0}\,[v]_{W^{s,p}(M)}^{\,p}
= \frac{D^{\,n+sp}}{\Vol(M)\,c_{M,s,p}}\,[u]_{W^{s,p}(M)}^{\,p}.
\]
Taking \(p\)-th roots yields
\[
\|u-u_M\|_{L^{p}(M)}
\le C\,[u]_{W^{s,p}(M)},\qquad
C = \Bigl(\frac{D^{\,n+sp}}{\Vol(M)\,c_{M,s,p}}\Bigr)^{1/p},
\]
where \(C\) depends only on \((M,g,s,p)\).  This completes the proof.
\end{proof}

\begin{proposition}\label{pro4.9}
	Let \(s\in(0,1)\) and \(p\in[1,\infty)\).
	\begin{enumerate}\itemsep0.25em
		\item If \(sp<n\), then \(W^{s,p}(M)\hookrightarrow L^{q}(M)\) continuously for every \(p\le q\le p_s^{*}\), and compactly for \(p\le q<p_s^{*}\).
		\item If \(sp=n\), then \(W^{s,p}(M)\hookrightarrow L^{q}(M)\) continuously for all \(q\in[p,\infty)\).
		\item If \(sp>n\), then \(W^{s,p}(M)\hookrightarrow C^{0,\alpha}(M)\) with \(\alpha=s-\frac{n}{p}\in(0,1)\).
	\end{enumerate}
\end{proposition}

\begin{proof}
Fix \(s\in(0,1)\) and \(p\in[1,\infty)\). Recall that
\[
[u]_{W^{s,p}(M)}^{p}
=\iint_{M\times M}|u(x)-u(y)|^{p}\,\Kps(x,y)\,d\mu(x)\,d\mu(y),
\]
and that by \eqref{eq4.4} there exist constants \(0<c\le C<\infty\) such that
\begin{equation}\label{eq4.13}
\frac{c}{\dist(x,y)^{\,n+sp}}
\le
\Kps(x,y)
\le
\frac{C}{\dist(x,y)^{\,n+sp}}
\qquad (x\neq y).
\end{equation}
Define the geodesic-distance Gagliardo seminorm
\[
[u]_{\mathrm{geo};s,p}^{p}
=\iint_{M\times M}\frac{|u(x)-u(y)|^{p}}{\dist(x,y)^{\,n+sp}}\,d\mu(x)\,d\mu(y).
\]
Then \eqref{eq4.13} yields the equivalences
\begin{equation}\label{eq4.14}
[u]_{W^{s,p}(M)}\simeq [u]_{\mathrm{geo};s,p},
\qquad
\|u\|_{W^{s,p}(M)}\simeq \|u\|_{L^{p}(M)}+[u]_{\mathrm{geo};s,p},
\end{equation}
with constants depending only on \((M,g,s,p)\).

Choose a finite smooth atlas \(\{(U_i,\psi_i)\}_{i=1}^{N}\), where each
\(\psi_i:U_i\to V_i\subset\mathbb{R}^{n}\) is bi-Lipschitz onto its image, and let
\(\{\eta_i\}_{i=1}^{N}\) be a smooth partition of unity subordinate to \(\{U_i\}\) with bounded overlap and uniformly bounded \(C^{1}\)-norm. Set
\[
u_i=\eta_i u,
\qquad
w_i=(u_i\circ\psi_i^{-1})\ \text{extended by }0\text{ to }\mathbb{R}^{n}.
\]
Standard coordinate estimates together with the fractional Leibniz rule (as in the proof of Proposition~\ref{pro4.6}) imply
\begin{equation}\label{eq4.15}
\sum_{i=1}^{N}\Bigl(\|w_i\|_{L^{p}(\mathbb{R}^{n})}^{p}
+[w_i]_{W^{s,p}(\mathbb{R}^{n})}^{p}\Bigr)
\le C_{0}\Bigl(\|u\|_{L^{p}(M)}^{p}+[u]_{\mathrm{geo};s,p}^{p}\Bigr),
\end{equation}
and for all \(q\ge p\),
\begin{equation}\label{eq4.16}
\|u\|_{L^{q}(M)}^{q}
\le C_{1}\sum_{i=1}^{N}\|w_i\|_{L^{q}(\mathbb{R}^{n})}^{q}.
\end{equation}
All constants depend only on \((M,g,s,p)\) (and on \(q\) when stated).

We use the classical Euclidean fractional embeddings on bounded Lipschitz domains
\(\Omega\subset\mathbb{R}^{n}\):
\begin{itemize}
	\item If \(sp<n\), then for all \(p\le q\le p_s^{*}=\frac{np}{n-sp}\),
	\begin{equation}\label{eq4.17}
		\|f\|_{L^{q}(\Omega)}
		\le C\bigl(\|f\|_{L^{p}(\Omega)}+[f]_{W^{s,p}(\Omega)}\bigr),
	\end{equation}
	and the embedding is compact for \(q<p_s^{*}\).
	\item If \(sp=n\), then for every \(q\in[p,\infty)\),
	\begin{equation}\label{eq4.18}
		\|f\|_{L^{q}(\Omega)}
		\le C(q)\bigl(\|f\|_{L^{p}(\Omega)}+[f]_{W^{s,p}(\Omega)}\bigr).
	\end{equation}
	\item If \(sp>n\), writing \(\alpha=s-\frac{n}{p}\in(0,1)\),
	\begin{equation}\label{eq4.19}
		|f(x)-f(y)|
		\le C\,|x-y|^{\alpha}[f]_{W^{s,p}(\Omega)}
		\qquad (x,y\in\Omega).
	\end{equation}
\end{itemize}

\textbf{Case \(sp<n\).}
Let \(p\le q\le p_s^{*}\). Applying \eqref{eq4.17} to each \(w_i\) and using \eqref{eq4.16},
\[
\|u\|_{L^{q}(M)}^{q}
\le C_{1}\sum_{i=1}^{N}\Bigl(\|w_i\|_{L^{p}}+[w_i]_{W^{s,p}}\Bigr)^{q}.
\]
Since \(q\ge p\) and \(N<\infty\), Hölder’s inequality and \eqref{eq4.15} give
\[
\|u\|_{L^{q}(M)}
\le C_{3}\bigl(\|u\|_{L^{p}(M)}+[u]_{\mathrm{geo};s,p}\bigr)
\simeq C_{4}\,\|u\|_{W^{s,p}(M)}.
\]
Thus \(W^{s,p}(M)\hookrightarrow L^{q}(M)\) for \(p\le q\le p_s^{*}\).

To obtain compactness for \(q<p_s^{*}\), let \((u_k)\) be bounded in \(W^{s,p}(M)\).  
By \eqref{eq4.15}, each \((w_{i,k})\) is bounded in \(W^{s,p}(V_i)\), hence (after passing to subsequences) converges strongly in \(L^{q}(V_i)\). A diagonal argument provides a subsequence such that \(w_{i,k}\to w_i\) for all \(i\).  
Using \eqref{eq4.16},
\[
\|u_k-u\|_{L^{q}(M)}^{q}
\le C_{1}\sum_{i=1}^{N}\|w_{i,k}-w_i\|_{L^{q}}^{q}\to0,
\]
so the embedding is compact.

\textbf{Case \(sp=n\).}
For any fixed \(q\in[p,\infty)\), applying \eqref{eq4.18} to each \(w_i\) and using \eqref{eq4.15}–\eqref{eq4.16} yields
\[
\|u\|_{L^{q}(M)}
\le C(q)\bigl(\|u\|_{L^{p}(M)}+[u]_{\mathrm{geo};s,p}\bigr)
\simeq C'(q)\,\|u\|_{W^{s,p}(M)}.
\]
Thus \(W^{s,p}(M)\hookrightarrow L^{q}(M)\) continuously for all \(q\in[p,\infty)\).

\textbf{Case \(sp>n\).}
Let \(\alpha=s-\frac{n}{p}\in(0,1)\).  
Applying \eqref{eq4.19} to each \(w_i\) on \(V_i\), and using bi-Lipschitz equivalence of \(|\cdot|\) and \(\dist(\cdot,\cdot)\),
\[
|u_i(x)-u_i(y)|
\le C\,\dist(x,y)^{\alpha}[w_i]_{W^{s,p}(V_i)}\qquad (x,y\in U_i),
\]
where \(u_i=\eta_i u\). Summing over \(i\) and noting the bounded overlap,
\[
|u(x)-u(y)|
\le C'\,\dist(x,y)^{\alpha}
\bigl(\|u\|_{L^{p}(M)}+[u]_{\mathrm{geo};s,p}\bigr),
\]
and hence
\[
[u]_{C^{0,\alpha}(M)}
\le C''\,\bigl(\|u\|_{L^{p}(M)}+[u]_{\mathrm{geo};s,p}\bigr)
\simeq C'''\,\|u\|_{W^{s,p}(M)}.
\]
Moreover, since \(M\) is compact,
\[
\|u\|_{L^\infty(M)}
\le |u_M|+\sup_{x\in M}|u(x)-u_M|
\le \Vol(M)^{-1/p}\|u\|_{L^p(M)}+\mathrm{diam}_g(M)^{\alpha}[u]_{C^{0,\alpha}(M)},
\]
so \(\|u\|_{C^{0,\alpha}(M)}\le C\|u\|_{W^{s,p}(M)}\).

This proves all three embedding statements.
\end{proof}

\begin{proposition}\label{pro4.10}
	If \(sp>n\) and \(1<p<\infty\), then \(W^{s,p}(M)\) is a Banach algebra:
	there exists \(C=C(M,g,s,p)>0\) such that for all \(u,v\in W^{s,p}(M)\),
	\[
	\|uv\|_{W^{s,p}(M)}\le C\,\|u\|_{W^{s,p}(M)}\,\|v\|_{W^{s,p}(M)}.
	\]
\end{proposition}

\begin{proof}
	Assume \(sp>n\) and \(1<p<\infty\). Let \(u,v\in W^{s,p}(M)\). We show that \(uv\in W^{s,p}(M)\) and derive the desired algebra estimate.

	By Proposition~\ref{pro4.9}(3), there exists \(\alpha=s-\frac{n}{p}\in(0,1)\) and a constant
	\(C_E=C_E(M,g,s,p)\) such that
	\[
	\|w\|_{C^{0,\alpha}(M)}\le C_E\,\|w\|_{W^{s,p}(M)}
	\qquad\forall\,w\in W^{s,p}(M).
	\]
	Since \(M\) is compact, this implies in particular
	\begin{equation}\label{eq4.20}
		\|w\|_{L^\infty(M)}\le C_E\,\|w\|_{W^{s,p}(M)}
		\qquad\forall\,w\in W^{s,p}(M).
	\end{equation}

	By \eqref{eq4.20} and H\"older's inequality,
	\begin{equation}\label{eq4.21}
		\|uv\|_{L^p(M)}
		\le \|u\|_{L^\infty(M)}\|v\|_{L^p(M)}
		\le C_E\,\|u\|_{W^{s,p}(M)}\,\|v\|_{W^{s,p}(M)}.
	\end{equation}

	Recall that
	\[
	[w]_{W^{s,p}(M)}^{p}
	=\iint_{M\times M}|w(x)-w(y)|^{p}\,\Kps(x,y)\,d\mu(x)\,d\mu(y),
	\]
	where \(\Kps\ge0\) is symmetric. For a.e.\ \((x,y)\in M\times M\),
	\[
	u(x)v(x)-u(y)v(y)
	=\bigl(u(x)-u(y)\bigr)v(x)
	+\bigl(v(x)-v(y)\bigr)u(y).
	\]
	Using \((a+b)^p\le 2^{p-1}(a^p+b^p)\),
	\[
	|u(x)v(x)-u(y)v(y)|^{p}
	\le 2^{p-1}\Bigl(
	|u(x)-u(y)|^{p}|v(x)|^{p}
	+|v(x)-v(y)|^{p}|u(y)|^{p}
	\Bigr).
	\]
	Integrating against \(\Kps(x,y)\,d\mu(x)\,d\mu(y)\) and applying Fubini,
	\begin{align*}
		[uv]_{W^{s,p}(M)}^{p}
		&\le 2^{p-1}
		\int_M |v(x)|^{p}
		\left(\int_M |u(x)-u(y)|^{p}\Kps(x,y)\,d\mu(y)\right)d\mu(x) \\
		&\quad
		+2^{p-1}
		\int_M |u(y)|^{p}
		\left(\int_M |v(x)-v(y)|^{p}\Kps(x,y)\,d\mu(x)\right)d\mu(y) \\
		&\le 2^{p-1}\|v\|_{L^\infty(M)}^{p}[u]_{W^{s,p}(M)}^{p}
		+2^{p-1}\|u\|_{L^\infty(M)}^{p}[v]_{W^{s,p}(M)}^{p}.
	\end{align*}
	Taking \(p\)-th roots and using \((a^{p}+b^{p})^{1/p}\le a+b\),
	\[
	[uv]_{W^{s,p}(M)}
	\le 2^{1-\frac1p}
	\Bigl(\|v\|_{L^\infty(M)}[u]_{W^{s,p}(M)}
	+\|u\|_{L^\infty(M)}[v]_{W^{s,p}(M)}\Bigr).
	\]
	Applying \eqref{eq4.20},
	\begin{equation}\label{eq4.22}
    \begin{aligned}
    [uv]_{W^{s,p}(M)}
		&\le 2^{1-\frac1p}C_E
		\Bigl(\|v\|_{W^{s,p}(M)}\|u\|_{W^{s,p}(M)}
		+\|u\|_{W^{s,p}(M)}\|v\|_{W^{s,p}(M)}\Bigr)\\
		&=2^{2-\frac1p}C_E\,\|u\|_{W^{s,p}(M)}\|v\|_{W^{s,p}(M)}.    
    \end{aligned}
	\end{equation}

	Using \(\|w\|_{W^{s,p}(M)}=\|w\|_{L^p(M)}+[w]_{W^{s,p}(M)}\) together with
	\eqref{eq4.21} and \eqref{eq4.22},
	\[
	\|uv\|_{W^{s,p}(M)}
	\le \|uv\|_{L^{p}(M)}+[uv]_{W^{s,p}(M)}
	\le \left(C_E+2^{2-\frac1p}C_E\right)
	\|u\|_{W^{s,p}(M)}\,\|v\|_{W^{s,p}(M)}.
	\]
	Thus \(uv\in W^{s,p}(M)\), and the algebra estimate holds with
	\[
	C=C_E\bigl(1+2^{2-\frac1p}\bigr).
	\]
\end{proof}

\subsection{The B--program: optimal \(L^p\)-term in fractional Sobolev inequalities}

In this subsection, for \(u\in W^{s,p}(M)\), we study the fractional Sobolev embedding in the following two equivalent forms:
\begin{equation}\label{eq4.23}
\begin{cases}
(I_{p,\mathrm{gen}}^{\,1}) &
\displaystyle \|u\|_{L^{p_s^*}(M)}
	\le A\,[u]_{W^{s,p}(M)} + B\,\|u\|_{L^{p}(M)},\\[1.2ex]
(I_{p,\mathrm{gen}}^{\,p}) &
\displaystyle \|u\|_{L^{p_s^*}(M)}^{p}
	\le A\,[u]_{W^{s,p}(M)}^{\,p}
	+ B\,\|u\|_{L^{p}(M)}^{\,p},
\end{cases}
\end{equation}
where \(A,B\ge 0\) are constants independent of \(u\).
The two inequalities \((I_{p,\mathrm{gen}}^{\,1})\) and \((I_{p,\mathrm{gen}}^{\,p})\) are equivalent up to a change of constants:
\((I_{p,\mathrm{gen}}^{\,1})\Rightarrow (I_{p,\mathrm{gen}}^{\,p})\) by \((a+b)^p\le 2^{p-1}(a^p+b^p)\),
and \((I_{p,\mathrm{gen}}^{\,p})\Rightarrow (I_{p,\mathrm{gen}}^{\,1})\) by \((a+b)^{1/p}\le a^{1/p}+b^{1/p}\).

Following Hebey \cite{Hebey1999}, we introduce
\begin{align*}
\mathcal{A}_p(M)
	&=\bigl\{A\in\mathbb{R}:\ \exists\,B\in\mathbb{R}\ \text{such that }(I_{p,\mathrm{gen}}^{\,1})\ \text{holds}\bigr\},\\
\mathcal{B}_p(M)
	&=\bigl\{B\in\mathbb{R}:\ \exists\,A\in\mathbb{R}\ \text{such that }(I_{p,\mathrm{gen}}^{\,1})\ \text{holds}\bigr\},
\end{align*}
and, for the \(p\)-power formulation,
\begin{align*}
\overline{\mathcal{A}_p(M)}
	&=\bigl\{A\in\mathbb{R}:\ \exists\,B\in\mathbb{R}\ \text{such that }(I_{p,\mathrm{gen}}^{\,p})\ \text{holds}\bigr\},\\
\overline{\mathcal{B}_p(M)}
	&=\bigl\{B\in\mathbb{R}:\ \exists\,A\in\mathbb{R}\ \text{such that }(I_{p,\mathrm{gen}}^{\,p})\ \text{holds}\bigr\}.
\end{align*}
The corresponding optimal constants are
\[
\alpha_p(M)=\inf\mathcal{A}_p(M),\qquad
\beta_p(M)=\inf\mathcal{B}_p(M),\qquad
\overline{\alpha_p(M)}=\inf\overline{\mathcal{A}_p(M)},\qquad
\overline{\beta_p(M)}=\inf\overline{\mathcal{B}_p(M)}.
\]

\begin{remark}\label{rem4.11}
We say that \(\mathcal{A}_p(M)\) is closed at the infimum if \(\alpha_p(M)\in\mathcal{A}_p(M)\).
Equivalently, there exists \(B\in\mathbb{R}\) such that
\begin{equation}\label{eq4.24}
(I_{p,\mathrm{opt}}^{\,1})\qquad
\|u\|_{L^{p_s^*}(M)}
	\le \alpha_p(M)\,[u]_{W^{s,p}(M)}
	+ B\,\|u\|_{L^{p}(M)}
	\qquad\forall\,u\in W^{s,p}(M).
\end{equation}
Similarly, \(\mathcal{B}_p(M)\) is closed at the infimum if \(\beta_p(M)\in\mathcal{B}_p(M)\); that is, if there exists \(A\in\mathbb{R}\) such that
\begin{equation}\label{eq4.25}
(J_{p,\mathrm{opt}}^{\,1})\qquad
\|u\|_{L^{p_s^*}(M)}
	\le A\,[u]_{W^{s,p}(M)}
	+ \beta_p(M)\,\|u\|_{L^{p}(M)}
	\qquad\forall\,u\in W^{s,p}(M).
\end{equation}
Analogous definitions apply to \(\overline{\mathcal{A}_p(M)}\) and \(\overline{\mathcal{B}_p(M)}\), replacing \((I_{p,\mathrm{gen}}^{\,1})\) with \((I_{p,\mathrm{gen}}^{\,p})\).
For example, closure at the infimum for \(\overline{\mathcal{B}_p(M)}\) means that there exists \(\overline{A}\in\mathbb{R}\) such that
\[
(J_{p,\mathrm{opt}}^{\,p})\qquad
\|u\|_{L^{p_s^*}(M)}^{p}
	\le \overline{A}\,[u]_{W^{s,p}(M)}^{\,p}
	+ \overline{\beta_p(M)}\,\|u\|_{L^{p}(M)}^{\,p}
	\qquad\forall\,u\in W^{s,p}(M).
\]
\end{remark}

In the remainder of this subsection we address the following questions with \(n>sp\):
\begin{enumerate}\itemsep0.2em
	\item Compute \(\beta_p(M)\) and \(\overline{\beta_p(M)}\) explicitly; equivalently, determine the optimal \(L^p\)-mass terms in the linear and \(p\)-power inequalities \eqref{eq4.23}.
	\item Prove that \(\mathcal{B}_p(M)\) and \(\overline{\mathcal{B}_p(M)}\) are closed at the infimum.
	\item Identify the precise range of exponents \(p\) for which the optimal inequality \((J_{p,\mathrm{opt}}^{\,p})\) holds on an arbitrary closed manifold \((M,g)\).
\end{enumerate}

First, we establish the validity of \eqref{eq4.23}.

\begin{lemma}\label{lem4.12}
Let \((M,g)\) be a closed \(n\)-dimensional Riemannian manifold, let \(s\in(0,1)\), and assume \(1\le p<\frac{n}{s}\).
Then there exist constants \(A,B>0\) such that for all \(u\in W^{s,p}(M)\),
\begin{equation}\label{eq4.26}
	\|u\|_{L^{p_s^*}(M)}
	\le
	A\Bigg(\iint_{M\times M}|u(x)-u(y)|^p\,\Kps(x,y)\,d\mu(x)\,d\mu(y)\Bigg)^{1/p}
	+ B\,\|u\|_{L^p(M)},
\end{equation}
where \(p_s^*=\frac{np}{n-sp}\).
\end{lemma}

\begin{proof}
Choose a finite family of normal coordinate charts \(\{(U_j,\phi_j)\}_{j=1}^N\) with uniformly controlled geometry, and let \(\{\eta_j\}_{j=1}^N\subset C^\infty(M)\) be a partition of unity subordinate to \(\{U_j\}\), with bounded overlap and uniformly bounded derivatives.
For each \(j\), define \(\Omega_j=\phi_j(U_j)\subset\mathbb{R}^n\) and
\[
v_j = (\eta_j u)\circ\phi_j^{-1}\quad\text{on }\Omega_j.
\]

Since \(\Omega_j\) is a bounded Lipschitz domain, there exists a bounded linear extension operator
\[
E_j:W^{s,p}(\Omega_j)\to W^{s,p}(\mathbb{R}^n),
\qquad \tilde v_j=E_j v_j.
\]
By the Euclidean fractional Sobolev inequality (see, e.g., \cite{DiNezza-Palatucci-Valdinoci2012}),
\[
\|\tilde v_j\|_{L^{p_s^*}(\mathbb{R}^n)}
\le
C\big([\tilde v_j]_{W^{s,p}(\mathbb{R}^n)}+\|\tilde v_j\|_{L^p(\mathbb{R}^n)}\big),
\]
with \(C=C(n,s,p)\).
Using the boundedness of \(E_j\) and the Jacobian and distance comparability in normal coordinates (cf. \cite{Rey-Saintier2024}), we obtain on each \(U_j\) that
\begin{equation}\label{eq4.27}
\|\eta_j u\|_{L^{p_s^*}(U_j)}
\le
C\big([\eta_j u]_{W^{s,p}(U_j)}+\|\eta_j u\|_{L^p(U_j)}\big),
\end{equation}
where \(C\) now depends only on \((M,g,s,p)\).

Summing \eqref{eq4.27} over \(j\) and using \(u=\sum_{j=1}^N \eta_j u\),
\[
\|u\|_{L^{p_s^*}(M)}
= \Big\|\sum_{j=1}^N \eta_j u\Big\|_{L^{p_s^*}(M)}
\le \sum_{j=1}^N \|\eta_j u\|_{L^{p_s^*}(U_j)}
\le C\sum_{j=1}^N\big([\eta_j u]_{W^{s,p}(U_j)} + \|\eta_j u\|_{L^p(U_j)}\big).
\]

The fractional Leibniz rule and the uniform bounds on \(\eta_j\) yield
\[
[\eta_j u]_{W^{s,p}(U_j)}
\le
C\big([u]_{W^{s,p}(U_j)} + \|u\|_{L^p(U_j)}\big),
\]
with \(C\) independent of \(j\).
Summing over \(j\) and using that \(\sum_{j=1}^N \|\eta_j u\|_{L^p(U_j)}\le C\|u\|_{L^p(M)}\), we obtain
\[
\sum_{j=1}^N [\eta_j u]_{W^{s,p}(U_j)}
\le
C\big([u]_{W^{s,p}(M)} + \|u\|_{L^p(M)}\big).
\]
Consequently,
\[
\|u\|_{L^{p_s^*}(M)}
\le
C\big([u]_{W^{s,p}(M)} + \|u\|_{L^p(M)}\big).
\]

Recalling Definition~\ref{def4.1}, we have
\[
[u]_{W^{s,p}(M)}^{p}
= \iint_{M\times M}|u(x)-u(y)|^p\,\Kps(x,y)\,d\mu(x)\,d\mu(y),
\]
so \eqref{eq4.26} follows by renaming constants.
\end{proof}

\begin{theorem}\label{thm4.13}
Let \((M,g)\) be a closed Riemannian \(n\)-manifold, let \(s\in(0,1)\), and assume \(1\le p<\frac{n}{s}\).
Then
\[
\beta_p(M)=\Vol(M)^{-s/n}.
\]
In particular, \(\mathcal{B}_p(M)\) is closed at the infimum: there exists \(A>0\) such that for all
\(u\in W^{s,p}(M)\),
\begin{equation}\label{eq4.28}
\|u\|_{L^{p_s^*}(M)}
\le
A\Bigg(\iint_{M\times M}|u(x)-u(y)|^{p}\,\Kps(x,y)\,d\mu(x)\,d\mu(y)\Bigg)^{1/p}
+\beta_p(M)\,\|u\|_{L^{p}(M)},
\end{equation}
where \(p_s^*=\frac{np}{n-sp}\). Thus the optimal inequality \((J_{p,\mathrm{opt}}^{\,1})\) holds.
\end{theorem}

\begin{proof}
Test \((I^{\,1}_{p,\mathrm{gen}})\) with the constant function \(u\equiv 1\).
Since \([1]_{W^{s,p}(M)}=0\), the inequality reduces to
\[
\|1\|_{L^{p_s^*}(M)}\le B\,\|1\|_{L^{p}(M)}
\quad\Longleftrightarrow\quad
\Vol(M)^{1/p_s^*}\le B\,\Vol(M)^{1/p}.
\]
Thus any admissible \(B\) must satisfy
\[
B\ge \Vol(M)^{1/p_s^*-1/p}
=\Vol(M)^{-s/n}.
\]
Hence
\[
\beta_p(M)\ge \Vol(M)^{-s/n}.
\]

Let \(u\in W^{s,p}(M)\), and denote its average by
\[
u_M=\frac{1}{\Vol(M)}\int_M u\,d\mu.
\]
Set \(v=u-u_M\), so \(v_M=0\).
By Proposition~\ref{pro4.8},
\[
\|v\|_{L^p(M)}
\le C_P\,[v]_{W^{s,p}(M)}
= C_P\,[u]_{W^{s,p}(M)}.
\]

By Lemma~\ref{lem4.12}, there exist constants \(A_0,B_0>0\) such that
\begin{equation}\label{eq4.29}
\|w\|_{L^{p_s^*}(M)}
\le
A_0\Bigg(\iint_{M\times M}|w(x)-w(y)|^{p}\,\Kps(x,y)\,d\mu(x)\,d\mu(y)\Bigg)^{1/p}
+B_0\,\|w\|_{L^p(M)}
\end{equation}
for all \(w\in W^{s,p}(M)\).
Applying \eqref{eq4.29} to \(w=v\) and using the Poincar\'e inequality,
\[
\|v\|_{L^{p_s^*}(M)}
\le A_0[v]_{W^{s,p}(M)} + B_0\|v\|_{L^{p}(M)}
\le (A_0 + B_0C_P)\,[u]_{W^{s,p}(M)}.
\]
Let \(A_1=A_0+B_0C_P\); then
\begin{equation}\label{eq4.30}
\|u-u_M\|_{L^{p_s^*}(M)}
\le
A_1\Bigg(\iint_{M\times M}|u(x)-u(y)|^p\,\Kps(x,y)\,d\mu(x)\,d\mu(y)\Bigg)^{1/p}.
\end{equation}

By H\"older's inequality,
\[
|u_M|
= \frac{1}{\Vol(M)}\Big|\int_M u\,d\mu\Big|
\le \Vol(M)^{-1/p}\|u\|_{L^p(M)}.
\]
Hence
\[
\|u_M\|_{L^{p_s^*}(M)}
= \Vol(M)^{1/p_s^*}\,|u_M|
\le \Vol(M)^{1/p_s^*-1/p}\,\|u\|_{L^p(M)}
= \Vol(M)^{-s/n}\,\|u\|_{L^p(M)}.
\]
Using \(u=(u-u_M)+u_M\) together with \eqref{eq4.30} and the previous estimate,
\[
\begin{aligned}
\|u\|_{L^{p_s^*}(M)}
&\le \|u-u_M\|_{L^{p_s^*}(M)} + \|u_M\|_{L^{p_s^*}(M)}\\
&\le A_1\Bigg(\iint_{M\times M}|u(x)-u(y)|^p\,\Kps(x,y)\,d\mu(x)\,d\mu(y)\Bigg)^{1/p}
+\Vol(M)^{-s/n}\,\|u\|_{L^p(M)}.
\end{aligned}
\]
Thus \((I^{\,1}_{p,\mathrm{gen}})\) holds with
\[
A=A_1,\qquad B=\Vol(M)^{-s/n},
\]
which shows
\[
\beta_p(M)\le \Vol(M)^{-s/n}.
\]

Combining the lower and upper bounds,
\[
\Vol(M)^{-s/n}\le \beta_p(M)\le \Vol(M)^{-s/n},
\]
we obtain
\[
\beta_p(M)=\Vol(M)^{-s/n}.
\]
Since \eqref{eq4.28} is valid with this constant, the set \(\mathcal{B}_p(M)\) is closed at the infimum.
\end{proof}

\begin{theorem}\label{thm4.14}
	Let $(M,g)$ be a closed Riemannian $n$-manifold, $s\in(0,1)$, $sp<n$ and $1\le p\le 2$.
	Then there exists $A=A(M,g,s,p)>0$ such that for every $u\in W^{s,p}(M)$,
	\begin{equation}\label{eq4.31}
	\Bigl(\int_M |u|^{p_s^*}\,d\mu\Bigr)^{\!p/p_s^*}
	\ \le\
	A \iint_{M\times M}\!|u(x)-u(y)|^p\,\Kps(x,y)\,d\mu(x)\,d\mu(y)
	\ +\ \Vol(M)^{-\,\frac{sp}{n}}\int_M |u|^p\,d\mu.
\end{equation}
	In particular, the optimal inequality \((J^{\,p}_{p,\mathrm{opt}})\) holds for all $n\ge2$ and all $p\in[1,2]$ with $sp<n$.
\end{theorem}

\begin{proof}
Let
\[
u_M=\frac{1}{\Vol(M)}\int_M u\,d\mu,
\qquad
p_s^*=\frac{np}{n-sp},
\qquad
V=\Vol(M).
\]
We distinguish the cases \(p_s^*\ge 2\) and \(p_s^*\le 2\).

Case A: \(p_s^*\ge 2\).
We first establish the Bakry-type convexity inequality: for all \(u\in L^{p_s^*}(M)\),
\begin{equation}\label{eq4.32}
\Bigl(\int_M |u|^{p_s^*}\,d\mu\Bigr)^{2/p_s^*}
\le
V^{-\,\frac{2(p_s^*-1)}{p_s^*}}\Bigl|\int_M u\,d\mu\Bigr|^2
+(p_s^*-1)\Bigl(\int_M |u-u_M|^{p_s^*}\,d\mu\Bigr)^{2/p_s^*}.
\end{equation}
If \(\int_M u\,d\mu=0\), then \(u_M=0\) and \eqref{eq4.32} is trivial since \(p_s^*-1\ge 1\).
Assume \(\int_M u\,d\mu\ne 0\). By homogeneity and by replacing \(u\) with \(-u\) if needed,
it suffices to consider \(u\in C^0(M)\) with \(\int_M u\,d\mu=V\), hence \(u_M=1\).

Write \(u=1+t v\) with \(t\ge 0\) and \(\int_M v\,d\mu=0\). Define
\[
\varphi(t)=\Bigl(\int_M |1+t v|^{p_s^*}\,d\mu\Bigr)^{2/p_s^*}.
\]
Since \(p_s^*\ge 2\), the map \(r\mapsto |r|^{p_s^*}\) is \(C^2\), hence \(\varphi\in C^2([0,\infty))\).
A direct computation gives \(\varphi(0)=V^{2/p_s^*}\) and \(\varphi'(0)=0\). Moreover,
\begin{align*}
\varphi''(t)
&=2p_s^*\Bigl(\frac{2}{p_s^*}-1\Bigr)
\Bigl(\int_M |1+t v|^{p_s^*}\,d\mu\Bigr)^{\frac{2}{p_s^*}-2}
\Bigl(\int_M |1+t v|^{p_s^*-1}\mathrm{sgn}(1+t v)\,v\,d\mu\Bigr)^2 \\
&\quad+2(p_s^*-1)\Bigl(\int_M |1+t v|^{p_s^*}\,d\mu\Bigr)^{\frac{2}{p_s^*}-1}
\int_M |1+t v|^{p_s^*-2}v^2\,d\mu.
\end{align*}
Since \(p_s^*\ge 2\), one has \(\frac{2}{p_s^*}-1\le 0\), hence the first term is nonpositive. By Hölder's inequality,
\[
\int_M |1+t v|^{p_s^*-2} v^2\,d\mu
\le
\Bigl(\int_M |1+t v|^{p_s^*}\,d\mu\Bigr)^{1-\frac{2}{p_s^*}}
\Bigl(\int_M |v|^{p_s^*}\,d\mu\Bigr)^{\frac{2}{p_s^*}}.
\]
Therefore,
\[
\varphi''(t)\le 2(p_s^*-1)\Bigl(\int_M |v|^{p_s^*}\,d\mu\Bigr)^{2/p_s^*}
\quad\text{for all }t\ge 0.
\]
Integrating twice and using \(\varphi'(0)=0\) yields
\[
\varphi(t)\le V^{2/p_s^*}+(p_s^*-1)t^2\Bigl(\int_M |v|^{p_s^*}\,d\mu\Bigr)^{2/p_s^*}.
\]
Taking \(t=1\) and recalling \(u=1+v\) gives \eqref{eq4.32}.

Raise both sides of \eqref{eq4.32} to the power \(p/2\in(0,1]\) and use
\((a+b)^{p/2}\le a^{p/2}+b^{p/2}\) for \(a,b\ge 0\):
\[
\Bigl(\int_M |u|^{p_s^*}\,d\mu\Bigr)^{p/p_s^*}
\le
V^{-\,\frac{(p_s^*-1)p}{p_s^*}}\Bigl|\int_M u\,d\mu\Bigr|^{p}
+(p_s^*-1)^{p/2}\Bigl(\int_M |u-u_M|^{p_s^*}\,d\mu\Bigr)^{p/p_s^*}.
\]
By Hölder's inequality,
\[
\Bigl|\int_M u\,d\mu\Bigr|
\le V^{1-\frac{1}{p}}\|u\|_{L^p(M)},
\]
hence
\[
V^{-\,\frac{(p_s^*-1)p}{p_s^*}}\Bigl|\int_M u\,d\mu\Bigr|^{p}
\le V^{(p-1)-\frac{(p_s^*-1)p}{p_s^*}} \|u\|_{L^p(M)}^p
=V^{-\,\frac{sp}{n}}\|u\|_{L^p(M)}^p,
\]
since \((p-1)-\frac{(p_s^*-1)p}{p_s^*}=-1+\frac{p}{p_s^*}=-\frac{sp}{n}\).

It remains to control \(\|u-u_M\|_{L^{p_s^*}(M)}^p\) by the energy.
Apply Lemma~\ref{lem4.12} to \(w=u-u_M\):
\[
\|u-u_M\|_{L^{p_s^*}(M)}
\le A_*\,[u]_{W^{s,p}(M)} + B_*\,\|u-u_M\|_{L^p(M)}.
\]
By Proposition~\ref{pro4.8}, \(\|u-u_M\|_{L^p(M)}\le C_P [u]_{W^{s,p}(M)}\), hence
\[
\|u-u_M\|_{L^{p_s^*}(M)}\le (A_*+B_*C_P)\,[u]_{W^{s,p}(M)}.
\]
Taking \(p\)-th powers and recalling
\(
[u]_{W^{s,p}(M)}^p=\iint_{M\times M}|u(x)-u(y)|^p\Kps(x,y)\,d\mu(x)\,d\mu(y)
\),
we obtain
\[
\Bigl(\int_M |u-u_M|^{p_s^*}\,d\mu\Bigr)^{p/p_s^*}
\le (A_*+B_*C_P)^p
\iint_{M\times M}|u(x)-u(y)|^p\Kps(x,y)\,d\mu(x)\,d\mu(y).
\]
Combining the last three displays yields \eqref{eq4.31} in Case A with
\[
A=(p_s^*-1)^{p/2}(A_*+B_*C_P)^p.
\]

Case B: \(p_s^*\le 2\).
If \(u_M=0\), then Lemma~\ref{lem4.12} and Proposition~\ref{pro4.8} give
\[
\Bigl(\int_M |u|^{p_s^*}\,d\mu\Bigr)^{p/p_s^*}
\le C_B
\iint_{M\times M}|u(x)-u(y)|^p\Kps(x,y)\,d\mu(x)\,d\mu(y),
\]
and \eqref{eq4.31} follows since the \(L^p\)-term has a nonnegative coefficient.

Assume \(u_M\ne 0\) and write \(u=u_M(1+v)\) with \(\int_M v\,d\mu=0\).
We claim that for every \(w\in L^{p_s^*}(M)\),
\begin{equation}\label{eq4.33}
\Bigl(\int_M |w|^{p_s^*}\,d\mu\Bigr)^{p/p_s^*}
\le
V^{\frac{p}{p_s^*}-p}\Bigl|\int_M w\,d\mu\Bigr|^p
+ C_{p_s^*} \Bigl(\int_M |w-w_M|^{p_s^*}\,d\mu\Bigr)^{p/p_s^*},
\end{equation}
where \(C_{p_s^*}>0\) depends only on \(p_s^*\).

To prove \eqref{eq4.33}, by homogeneity it suffices to take \(w=1+v\) with \(\int_M v\,d\mu=0\).
For \(p_s^*\in(1,2]\) we use
\[
(1+x)^{p_s^*}\le 1+p_s^* x+x^{p_s^*}\ (x\ge 0),
\qquad
(1-x)^{p_s^*}\le 1-p_s^* x+x^{p_s^*}\ (0\le x\le 1),
\qquad
(x-1)^{p_s^*}\le x^{p_s^*}\ (x\ge 1).
\]
Decompose \(M=A\cup B\cup C\) with
\[
A=\{v\ge 0\},
\qquad
B=\{-1\le v<0\},
\qquad
C=\{v<-1\}.
\]
Then
\begin{align*}
\int_M |1+v|^{p_s^*}\,d\mu
&\le \mu(\{v\ge -1\})+\int_M |v|^{p_s^*}\,d\mu + p_s^*\int_{\{v<-1\}} |v|\,d\mu.
\end{align*}
By Hölder,
\[
\int_{\{v<-1\}} |v|\,d\mu
\le \|v\|_{L^{p_s^*}(M)}\,\mu(\{v<-1\})^{1-\frac{1}{p_s^*}}.
\]
Set \(X_0=\mu(\{v\ge -1\})\in[0,V]\), \(t=V-X_0\), and \(s=\|v\|_{L^{p_s^*}(M)}^{p_s^*}\). Then
\[
\int_M |1+v|^{p_s^*}\,d\mu
\le V+s+\Bigl(-t+p_s^* s^{1/p_s^*} t^{(p_s^*-1)/p_s^*}\Bigr).
\]
A direct maximization in \(t\ge 0\) gives
\[
\sup_{t\ge 0}\Bigl(-t+p_s^* s^{1/p_s^*} t^{(p_s^*-1)/p_s^*}\Bigr)
=(p_s^*-1)^{p_s^*-1}s.
\]
Therefore,
\[
\int_M |1+v|^{p_s^*}\,d\mu
\le V+\Bigl(1+(p_s^*-1)^{p_s^*-1}\Bigr)\int_M |v|^{p_s^*}\,d\mu.
\]
Since \(p/p_s^*\le 1\), \((a+b)^{p/p_s^*}\le a^{p/p_s^*}+b^{p/p_s^*}\) for \(a,b\ge 0\), hence
\[
\Bigl(\int_M |1+v|^{p_s^*}\,d\mu\Bigr)^{p/p_s^*}
\le V^{p/p_s^*}
+\Bigl(1+(p_s^*-1)^{p_s^*-1}\Bigr)^{p/p_s^*}
\Bigl(\int_M |v|^{p_s^*}\,d\mu\Bigr)^{p/p_s^*}.
\]
This is \eqref{eq4.33} with
\[
C_{p_s^*}=\Bigl(1+(p_s^*-1)^{p_s^*-1}\Bigr)^{p/p_s^*}.
\]
The general \(w\) follows by writing \(w=w_M(1+v)\).

Apply \eqref{eq4.33} to \(w=u\). Since \(u_M=\frac{1}{V}\int_M u\,d\mu\),
\[
V^{p/p_s^*}|u_M|^p=V^{\frac{p}{p_s^*}-p}\Bigl|\int_M u\,d\mu\Bigr|^p.
\]
Moreover, by Hölder,
\[
|u_M|^p\le V^{-1}\|u\|_{L^p(M)}^p,
\qquad
V^{p/p_s^*}|u_M|^p\le V^{\frac{p}{p_s^*}-1}\|u\|_{L^p(M)}^p
=V^{-\,\frac{sp}{n}}\|u\|_{L^p(M)}^p.
\]
Finally, Lemma~\ref{lem4.12} and Proposition~\ref{pro4.8} give
\[
\Bigl(\int_M |u-u_M|^{p_s^*}\,d\mu\Bigr)^{p/p_s^*}
\le (A_*+B_*C_P)^p
\iint_{M\times M}|u(x)-u(y)|^p\Kps(x,y)\,d\mu(x)\,d\mu(y).
\]
Combining these estimates yields \eqref{eq4.31} in Case B with
\[
A=C_{p_s^*}(A_*+B_*C_P)^p.
\]

This completes the proof of \eqref{eq4.31}. Testing \eqref{eq4.31} with \(u\equiv 1\) shows that the coefficient
\(V^{-sp/n}\) in front of \(\int_M |u|^p\,d\mu\) is optimal.
\end{proof}

\begin{theorem}\label{thm4.15}
Let $(M,g)$ be a closed Riemannian $n$--manifold with $n\ge3$, let $s\in(0,1)$, and let $p\in(2,n)$ satisfy $sp<n$.  
Then the optimal inequality $(J^{\,p}_{p,\mathrm{opt}})$ cannot hold for all $u\in W^{s,p}(M)$.
\end{theorem}

\begin{proof}
Fix a nonconstant function \(u\in C^\infty(M)\). Set
\[
V=\Vol(M),\qquad 
m_1=\int_M u\,d\mu,\qquad 
m_2=\int_M u^2\,d\mu.
\]
Since \(M\) is compact, \(\|u\|_{L^\infty(M)}<\infty\). Choose
\[
\varepsilon_0=\frac{1}{2\|u\|_{L^\infty(M)}}>0.
\]
For \(0<\varepsilon<\varepsilon_0\) define \(u_\varepsilon=1+\varepsilon u\). Then
\(u_\varepsilon(x)\ge \frac12\) on \(M\), hence \(|u_\varepsilon|^t=u_\varepsilon^t\) for every \(t>0\).

Let \(t>2\). By Taylor's formula for the \(C^2\) function \(r\mapsto r^t\) around \(r=1\),
\[
(1+\varepsilon u)^t
=1+t\varepsilon u+\frac{t(t-1)}{2}\varepsilon^2 u^2 + \varepsilon^2 r_{t,\varepsilon}(x),
\]
where \(r_{t,\varepsilon}\to 0\) uniformly on \(M\) as \(\varepsilon\to 0\).
Integrating over \(M\) yields
\begin{equation}\label{eq4.34}
\int_M |1+\varepsilon u|^{t}\,d\mu
=V+t\,\varepsilon\,m_1+\frac{t(t-1)}{2}\varepsilon^2\,m_2+o(\varepsilon^2)
\qquad (\varepsilon\to0).
\end{equation}

Applying \eqref{eq4.34} with \(t=p\) gives
\begin{equation}\label{eq4.35}
\int_M |u_\varepsilon|^{p}\,d\mu
=V+p\varepsilon m_1+\frac{p(p-1)}{2}\varepsilon^2 m_2+o(\varepsilon^2).
\end{equation}
Applying \eqref{eq4.34} with \(t=p_s^*\) and then raising to \(\alpha=p/p_s^*\in(0,1)\), we use
\[
(V+a\varepsilon+b\varepsilon^2)^\alpha
=
V^\alpha+\alpha V^{\alpha-1}a\varepsilon
+\Bigl(\alpha V^{\alpha-1}b+\frac{\alpha(\alpha-1)}{2}V^{\alpha-2}a^2\Bigr)\varepsilon^2
+o(\varepsilon^2),
\]
with \(a=p_s^* m_1\), \(b=\frac{p_s^*(p_s^*-1)}{2}m_2\), to obtain
\begin{equation}\label{eq4.36}
\Bigl(\int_M |u_\varepsilon|^{p_s^*}\,d\mu\Bigr)^{p/p_s^*}
=V^{p/p_s^*}
+p\,V^{\frac{p}{p_s^*}-1}\varepsilon m_1
+\Bigl[\frac{p(p_s^*-1)}{2}V^{\frac{p}{p_s^*}-1}m_2
+\frac{p(p-p_s^*)}{2}V^{\frac{p}{p_s^*}-2}m_1^2\Bigr]\varepsilon^2
+o(\varepsilon^2).
\end{equation}

For the Gagliardo term, since \(u_\varepsilon(x)-u_\varepsilon(y)=\varepsilon(u(x)-u(y))\),
\[
\iint_{M\times M}|u_\varepsilon(x)-u_\varepsilon(y)|^p\,\Kps(x,y)\,d\mu(x)\,d\mu(y)
=\varepsilon^p
\iint_{M\times M}|u(x)-u(y)|^p\,\Kps(x,y)\,d\mu(x)\,d\mu(y).
\]
Because \(p>2\), we have \(\varepsilon^p=o(\varepsilon^2)\) as \(\varepsilon\to 0\), hence the above term is \(o(\varepsilon^2)\).

Assume, toward a contradiction, that \((J^{\,p}_{p,\mathrm{opt}})\) holds for all \(u\in W^{s,p}(M)\), namely
\[
\Bigl(\int_M |u|^{p_s^*}\,d\mu\Bigr)^{p/p_s^*}
\le
A
\iint_{M\times M}|u(x)-u(y)|^p\,\Kps(x,y)\,d\mu(x)\,d\mu(y)
+
V^{-sp/n}\int_M |u|^p\,d\mu
\]
for some fixed constant \(A\in\mathbb{R}\).

Apply this inequality to \(u_\varepsilon\) and insert \eqref{eq4.35}--\eqref{eq4.36}. Using the estimate above for the
Gagliardo term, we obtain
\[
\Bigl(\int_M |u_\varepsilon|^{p_s^*}\,d\mu\Bigr)^{p/p_s^*}
\le
A\,o(\varepsilon^2)
+
V^{-sp/n}\int_M |u_\varepsilon|^p\,d\mu.
\]
Since \(\frac{p}{p_s^*}=1-\frac{sp}{n}\), the constant and linear terms in \(\varepsilon\) match identically.
Comparing the coefficients of \(\varepsilon^2\) yields
\[
\frac{p(p_s^*-1)}{2}V^{\frac{p}{p_s^*}-1}m_2
+\frac{p(p-p_s^*)}{2}V^{\frac{p}{p_s^*}-2}m_1^2
\le
\frac{p(p-1)}{2}V^{-sp/n}m_2.
\]
Using \(V^{\frac{p}{p_s^*}-1}=V^{-sp/n}\) and \(V^{\frac{p}{p_s^*}-2}=V^{-1-sp/n}\), this becomes
\[
(p_s^*-1)m_2+(p-p_s^*)V^{-1}m_1^2\le (p-1)m_2,
\]
that is,
\[
(p_s^*-p)m_2\le (p_s^*-p)V^{-1}m_1^2.
\]
Since \(sp<n\) implies \(p_s^*>p\), we have \(p_s^*-p>0\), and hence
\[
\int_M u^2\,d\mu\le \frac{1}{V}\Bigl(\int_M u\,d\mu\Bigr)^2.
\]
By Cauchy--Schwarz,
\[
\int_M u^2\,d\mu\ge \frac{1}{V}\Bigl(\int_M u\,d\mu\Bigr)^2,
\]
with equality if and only if \(u\) is constant. This contradicts the choice of \(u\) nonconstant.

Therefore the optimal inequality \((J^{\,p}_{p,\mathrm{opt}})\) cannot hold for all \(u\in W^{s,p}(M)\).
\end{proof}

\subsection{The A--program: optimal and improved leading coefficients}

In this subsection, we pursue the following two goals.
\begin{enumerate}\itemsep0.2em
\item We establish an almost sharp fractional Sobolev embedding
\[
W^{s,p}(M)\hookrightarrow L^{p_s^*}(M),
\qquad
p_s^*=\frac{np}{n-sp},
\]
in the precise form stated in Theorem~\ref{thm4.16}.
\item As a consequence of the almost sharp inequality, we derive an improved fractional Sobolev inequality under the constraint \eqref{eq4.41}.
\end{enumerate}

Let \(K(n,s,p)\) be the sharp constant in the Euclidean embedding
\(W^{s,p}(\mathbb{R}^n)\hookrightarrow L^{p_s^*}(\mathbb{R}^n)\), namely the smallest constant such that
\[
\|u\|_{L^{p_s^*}(\mathbb{R}^n)}^{p}\le K(n,s,p)\,[u]_{s,p}^{p}
\qquad\text{for all }u\in W^{s,p}(\mathbb{R}^n).
\]
Equivalently,
\[
K(n,s,p)^{-1}
=\inf_{0\neq u\in W^{s,p}(\mathbb{R}^n)}
\frac{[u]_{s,p}^p}{\|u\|_{L^{p_s^*}(\mathbb{R}^n)}^{p}},
\]
where
\[
[u]_{s,p}^p=\iint_{\mathbb{R}^n\times\mathbb{R}^n}
\frac{|u(x)-u(y)|^p}{|x-y|^{n+sp}}\,dx\,dy.
\]

In the same spirit as \cite{Hang2022, Yan2023}, we obtain an almost sharp fractional Sobolev embedding on \((M,g)\) by means of the concentration--compactness principle.

\begin{theorem}\label{thm4.16}
Let $(M,g)$ be a closed $n$-dimensional Riemannian manifold. Let $s\in(0,1)$ and $p\in(1,\infty)$ with $n>sp$.
Then, for every $\varepsilon>0$ there exists a constant $B=B(M,g,s,p,\varepsilon)>0$ such that for all $u\in W^{s,p}(M)$,
\begin{equation}\label{eq4.37}
  \Bigl(\int_M |u|^{p_s^*}\,d\mu\Bigr)^{\!\frac{p}{p_s^*}}
\le
\bigl(K(n,s,p)+\varepsilon\bigr)\iint_{M\times M} |u(x)-u(y)|^{p}\,\Kps(x,y)d\mu(x)d\mu(y)
+B\int_M |u|^pd\mu.  
\end{equation}
\end{theorem}

\begin{proof}
Set \(\alpha=K(n,s,p)+\varepsilon\).
Assume by contradiction that \eqref{eq4.37} is false. Then for each \(j\in\mathbb N\) there exists
\(u_j\in W^{s,p}(M)\) such that
\[
\Bigl(\int_M |u_j|^{p_s^*}\,d\mu\Bigr)^{\frac{p}{p_s^*}}
>
\alpha\iint_{M\times M}|u_j(x)-u_j(y)|^p\Kps(x,y)d\mu(x)d\mu(y)
+j\int_M |u_j|^p\,d\mu .
\]
By scaling we may assume
\[
\int_M |u_j|^{p_s^*}\,d\mu=1.
\]
Then
\begin{equation}\label{eq4.37a}
\iint_{M\times M}|u_j(x)-u_j(y)|^p\Kps(x,y)d\mu(x)d\mu(y)<\frac{1}{\alpha},
\qquad
\int_M |u_j|^p\,d\mu<\frac{1}{j}.
\end{equation}
In particular \((u_j)\) is bounded in \(W^{s,p}(M)\) and \(u_j\to0\) strongly in \(L^p(M)\).
Since \(1<p<\infty\), \(W^{s,p}(M)\) is reflexive, hence up to a subsequence
\(u_j\rightharpoonup u\) weakly in \(W^{s,p}(M)\). The strong \(L^p\) convergence forces \(u=0\), so
\[
u_j\rightharpoonup0 \quad\text{weakly in }W^{s,p}(M).
\]

Define finite Borel measures on \(M\) by
\[
\nu_j=|u_j|^{p_s^*}d\mu,
\qquad
\sigma_j=\Bigl(\int_M |u_j(x)-u_j(y)|^p\Kps(x,y)d\mu(y)\Bigr)d\mu(x).
\]
Then \(\nu_j(M)=1\) and \(\sigma_j(M)\) equals the energy in \eqref{eq4.37a}, hence \(\sigma_j(M)<1/\alpha\).
By compactness of \(M\), after passing to a subsequence we have weak-* convergence of measures
\[
\nu_j\rightharpoonup\nu,
\qquad
\sigma_j\rightharpoonup\sigma
\quad\text{in }\mathcal M(M).
\]
In particular,
\begin{equation}\label{eq4.39}
\nu(M)=1,
\qquad
\sigma(M)\le\frac{1}{\alpha}.
\end{equation}

We next record the localized inequality that links \(\nu\) and \(\sigma\).
Fix \(\delta>0\). By normal coordinates and the Euclidean sharp inequality with constant \(K(n,s,p)\),
there exists \(r_\delta>0\) such that for every \(\phi\in C^\infty(M)\) with \(\mathrm{supp}\,\phi\subset B_{r_\delta}(x_0)\)
one has
\begin{equation}\label{eq4.loc}
\Bigl(\int_M |\phi u|^{p_s^*}d\mu\Bigr)^{\frac{p}{p_s^*}}
\le (K(n,s,p)+\delta)\iint_{M\times M}|\phi(x)u(x)-\phi(y)u(y)|^p\Kps(x,y)d\mu(x)d\mu(y)
+C_{\delta,\phi}\int_M |u|^p\,d\mu
\end{equation}
for all \(u\in W^{s,p}(M)\), where \(C_{\delta,\phi}<\infty\).

Apply \eqref{eq4.loc} to \(u=u_j\). Using the inequality
\[
|\phi(x)u_j(x)-\phi(y)u_j(y)|^p
\le C\bigl(|u_j(x)-u_j(y)|^p+|u_j(y)|^p|\phi(x)-\phi(y)|^p\bigr),
\]
the second term contributes at most \(C'_{\phi}\|u_j\|_{L^p(M)}^p\), hence tends to \(0\) by \eqref{eq4.37a}.
Therefore, passing to the limit \(j\to\infty\) in \eqref{eq4.loc} and then letting \(\delta\to0\), we obtain
\begin{equation}\label{eq4.key}
\Bigl(\int_M |\phi|^{p_s^*}d\nu\Bigr)^{\frac{p}{p_s^*}}
\le K(n,s,p)\int_M |\phi|^p\,d\sigma
\qquad\text{for all }\phi\in C^\infty(M).
\end{equation}

The estimate \eqref{eq4.key} implies the concentration--compactness decomposition \cite{Lions1984-1,Lions1984-2}: there exist at most countably many
points \(\{x_i\}\subset M\) and numbers \(\nu_i,\sigma_i\ge0\) such that
\begin{equation}\label{eq4.38}
\nu=\sum_i \nu_i\delta_{x_i},
\qquad
\sigma\ge\sum_i \sigma_i\delta_{x_i},
\qquad
\nu_i^{\frac{p}{p_s^*}}\le K(n,s,p)\sigma_i.
\end{equation}

Since \(\theta=\frac{p}{p_s^*}\in(0,1)\), we have \((a+b)^\theta\le a^\theta+b^\theta\) for \(a,b\ge0\). Using \eqref{eq4.38} and \eqref{eq4.39},
\[
1=\nu(M)^\theta=\Bigl(\sum_i \nu_i\Bigr)^\theta
\le\sum_i \nu_i^\theta
\le K(n,s,p)\sum_i \sigma_i
\le K(n,s,p)\sigma(M)
\le \frac{K(n,s,p)}{\alpha}<1,
\]
a contradiction. Hence \eqref{eq4.37} holds.
\end{proof}

\begin{remark}\label{rem4.17}
In Theorem~\ref{thm4.16}, it is natural to ask whether the following sharp fractional Sobolev inequality holds:
\begin{equation}\label{eq4.40}
\Bigl(\int_M |u|^{p_s^*}\,d\mu\Bigr)^{\frac{p}{p_s^*}}
\le
K(n,s,p)\iint_{M\times M} |u(x)-u(y)|^{p}\Kps(x,y)\,d\mu(x)\,d\mu(y)
+
B\int_M |u|^{p}\,d\mu,
\qquad u\in W^{s,p}(M).
\end{equation}
This would be a fractional analogue of the main result in \cite{Zeitler2025}.

A natural idea, at least when \(p=2\), is to try to adapt the integer-order argument based on a shifted operator \(-\Delta_g+\alpha\) and an identity of the form
\[
\langle (-\Delta_g)u,u\rangle_{L^2}
=
\langle (-\Delta_g+\alpha)u,u\rangle_{L^2}
-\alpha\|u\|_{L^2}^2.
\]
However, for \(s\in(0,1)\) the map \(\lambda\mapsto \lambda^{s}\) is concave on \((0,\infty)\), and one has for every \(\lambda\ge0\) and \(\alpha>0\),
\[
(\lambda+\alpha)^s-\alpha^s \le \lambda^s.
\]
By spectral calculus this yields the inequality, for \(u\in C^\infty(M)\),
\[
\langle (-\Delta_g+\alpha)^s u,u\rangle_{L^2}
-\alpha^s\|u\|_{L^2}^2
\le
\langle (-\Delta_g)^s u,u\rangle_{L^2},
\]
which is the opposite direction from the exact linearization available at \(s=1\).
This lack of a suitable linearization mechanism is one of the obstructions to extending the classical (integer-order) argument to \eqref{eq4.40}.
\end{remark}

\begin{theorem}\label{thm4.18}
Let $(M,g)$ be a closed $n$-dimensional Riemannian manifold. Let $s\in(0,1)$ and $p\in(1,\infty)$ with $n>sp$.
Let $f_i\in C^1(M)$, $i=1,\ldots,N$, be sign-changing functions satisfying
\[
\sum_{i=1}^N |f_i|^p \equiv 1 \text{ on } M,
\]
and assume the orthogonality conditions
\begin{equation}\label{eq4.41}
\int_M f_i |f_i|^{p_s^*-1} |u|^{p_s^*} d\mu = 0,
\qquad i=1,\ldots,N.
\end{equation}
Then for every $\varepsilon>0$ there exists a constant
$B=B(M,g,s,\{f_i\},\varepsilon)>0$ such that for all $u\in W^{s,p}(M)$,
\begin{equation}\label{eq4.42}
\Bigl(\int_M |u|^{p_s^*} d\mu\Bigr)^{\frac{p}{p_s^*}}
\le
\Bigl(\frac{K(n,s,p)}{2^{sp/n}}+\varepsilon\Bigr)
\iint_{M\times M} |u(x)-u(y)|^p \Kps(x,y) d\mu(x) d\mu(y)
+ B\int_M |u|^p d\mu.
\end{equation}
\end{theorem}

\begin{proof}
For each $i$, set $f_{i,+}=\max\{f_i,0\}$ and $f_{i,-}=\max\{-f_i,0\}$, so that
$f_i=f_{i,+}-f_{i,-}$, $|f_i|=f_{i,+}+f_{i,-}$, and $f_{i,+} f_{i,-}\equiv 0$.
Since $f_i\in C^1(M)$, both $f_{i,+}$ and $f_{i,-}$ are Lipschitz on $M$.

From \eqref{eq4.41} we get
\[
0=\int_M \bigl((f_{i,+})^{p_s^*}-(f_{i,-})^{p_s^*}\bigr)|u|^{p_s^*} d\mu,
\]
hence
\[
A_i=\int_M (f_{i,+})^{p_s^*}|u|^{p_s^*} d\mu
=
\int_M (f_{i,-})^{p_s^*}|u|^{p_s^*} d\mu
=B_i.
\]
Since $\frac{p}{p_s^*}=1-\frac{sp}{n}\in(0,1)$, we have the identity
\begin{equation}\label{eq4.43}
\|f_i u\|_{L^{p_s^*}(M)}^p
=
(A_i+B_i)^{\frac{p}{p_s^*}}
=
(2A_i)^{\frac{p}{p_s^*}}
=
2^{-sp/n}\bigl(\|f_{i,+}u\|_{L^{p_s^*}(M)}^p+\|f_{i,-}u\|_{L^{p_s^*}(M)}^p\bigr).
\end{equation}

Fix $\varepsilon>0$. Apply Theorem~\ref{thm4.16} to $f_{i,+}u$ and $f_{i,-}u$ with a parameter
$\varepsilon_1>0$ to be chosen later. For each $i$ and $\sigma\in\{+,-\}$ we obtain
\[
\|f_{i,\sigma}u\|_{L^{p_s^*}(M)}^p
\le
\bigl(K(n,s,p)+\varepsilon_1\bigr)
\iint_{M\times M} |f_{i,\sigma}(x)u(x)-f_{i,\sigma}(y)u(y)|^p \Kps(x,y) d\mu(x)d\mu(y)
+ B_{i,\sigma}\int_M |u|^p d\mu,
\]
where $B_{i,\sigma}=B(M,g,s,p,\varepsilon_1,f_{i,\sigma})$ and we used $|f_{i,\sigma}u|^p\le \|f_i\|_{L^\infty}^p |u|^p$.

Insert these bounds into \eqref{eq4.43}. It remains to estimate the product energy.
Fix a Lipschitz function $f$ and write
\[
f(x)u(x)-f(y)u(y)=f(x)(u(x)-u(y))+(f(x)-f(y))u(y).
\]
For any $\delta\in(0,1)$, Young's inequality gives
\[
|a+b|^p \le (1+\delta)^{p-1}|a|^p+\Bigl(1+\frac{1}{\delta}\Bigr)^{p-1}|b|^p.
\]
Applying this with $a=f(x)(u(x)-u(y))$ and $b=(f(x)-f(y))u(y)$ yields
\[
|f(x)u(x)-f(y)u(y)|^p
\le
(1+\delta)^{p-1}|f(x)|^p|u(x)-u(y)|^p
+\Bigl(1+\frac{1}{\delta}\Bigr)^{p-1}|u(y)|^p|f(x)-f(y)|^p.
\]
Integrating against $\Kps(x,y)\,d\mu(x)d\mu(y)$ and using Fubini, we obtain
\begin{equation}\label{eq4.45}
\begin{aligned}
&\iint_{M\times M} |f(x)u(x)-f(y)u(y)|^p \Kps(x,y)\,d\mu(x)d\mu(y)\\
&\le
(1+\delta)^{p-1}\iint_{M\times M} |f(x)|^p |u(x)-u(y)|^p \Kps(x,y)\,d\mu(x)d\mu(y)
+ C(f,\delta)\int_M |u|^p d\mu,
\end{aligned}
\end{equation}
where $C(f,\delta)<\infty$ is obtained as follows: by the upper bound
$\Kps(x,y)\le C\,\dist(x,y)^{-n-sp}$ and the Lipschitz bound $|f(x)-f(y)|\le L_f\dist(x,y)$,
\[
\sup_{y\in M}\int_M |f(x)-f(y)|^p \Kps(x,y)\,d\mu(x)
\le
C L_f^p \sup_{y\in M}\int_M \dist(x,y)^{p-n-sp}\,d\mu(x)
<\infty,
\]
since $p-sp>0$.

Apply \eqref{eq4.45} with $f=f_{i,\sigma}$, sum over $\sigma\in\{+,-\}$, and use
$|f_{i,+}|^p+|f_{i,-}|^p=|f_i|^p$ to deduce
\[
\|f_i u\|_{L^{p_s^*}(M)}^p
\le
2^{-sp/n}\bigl(K(n,s,p)+\varepsilon_1\bigr)(1+\delta)^{p-1}
\iint_{M\times M} |f_i(x)|^p|u(x)-u(y)|^p \Kps(x,y)\,d\mu(x)d\mu(y)
+ \widetilde B_i\int_M |u|^p d\mu.
\]
Summing over $i=1,\ldots,N$ and using $\sum_i |f_i|^p\equiv 1$ yields
\begin{equation}\label{eq4.46}
\sum_{i=1}^N \|f_i u\|_{L^{p_s^*}(M)}^p
\le
2^{-sp/n}\bigl(K(n,s,p)+\varepsilon_1\bigr)(1+\delta)^{p-1}
\iint_{M\times M} |u(x)-u(y)|^p \Kps(x,y)\,d\mu(x)d\mu(y)
+ B_0\int_M |u|^p d\mu.
\end{equation}

Since $p_s^*>p$, the exponent $\frac{p_s^*}{p}>1$ and the triangle inequality in $L^{p_s^*/p}(M)$ gives
\[
\|u\|_{L^{p_s^*}(M)}^p
=\||u|^p\|_{L^{p_s^*/p}(M)}
=\Bigl\|\sum_{i=1}^N |f_i|^p|u|^p\Bigr\|_{L^{p_s^*/p}(M)}
\le
\sum_{i=1}^N \||f_i|^p|u|^p\|_{L^{p_s^*/p}(M)}
=
\sum_{i=1}^N \|f_i u\|_{L^{p_s^*}(M)}^p.
\]
Combining this with \eqref{eq4.46} we arrive at
\[
\|u\|_{L^{p_s^*}(M)}^p
\le
2^{-sp/n}\bigl(K(n,s,p)+\varepsilon_1\bigr)(1+\delta)^{p-1}
\iint_{M\times M} |u(x)-u(y)|^p \Kps(x,y)\,d\mu(x)d\mu(y)
+ B_0\int_M |u|^p d\mu.
\]

Finally choose $\varepsilon_1>0$ and $\delta\in(0,1)$ so small that
\[
2^{-sp/n}\bigl(K(n,s,p)+\varepsilon_1\bigr)(1+\delta)^{p-1}
\le
\frac{K(n,s,p)}{2^{sp/n}}+\varepsilon.
\]
Absorbing constants into a new $B$ yields \eqref{eq4.42}.
\end{proof}

\section*{Acknowledgments}
%We would like to thank the anonymous referee for his/her careful readings of our manuscript and the useful comments. 
Z. Yang is supported by National Natural Science Foundation of China (12301145, 12261107, 12561020) and Yunnan Fundamental Research Projects (202301AU070144, 202401AU070123).

\medskip
{\bf Data availability:}  Data sharing is not applicable to this article as no new data were created or analyzed in this study.

\medskip
{\bf Conflict of Interests:} The Author declares that there is no conflict of interest.

\bibliographystyle{plain}
\bibliography{ref}

@article {DiNezza-Palatucci-Valdinoci2012,
    AUTHOR = {Di Nezza, E. and Palatucci, G. and Valdinoci,
              E.},
     TITLE = {Hitchhiker's guide to the fractional {S}obolev spaces},
   JOURNAL = {Bull. Sci. Math.},
  FJOURNAL = {Bulletin des Sciences Math\'{e}matiques},
    VOLUME = {136},
      YEAR = {2012},
    NUMBER = {5},
     PAGES = {521--573},
      ISSN = {0007-4497},
   MRCLASS = {46E35 (35A23 35S05 35S30)},
  MRNUMBER = {2944369},
MRREVIEWER = {Lanzhe Liu},
       DOI = {10.1016/j.bulsci.2011.12.004},
       URL = {https://doi.org/10.1016/j.bulsci.2011.12.004},
}

@book {Davies1989,
	AUTHOR = {Davies, E. B.},
	TITLE = {Heat kernels and spectral theory},
	SERIES = {Cambridge Tracts in Mathematics},
	VOLUME = {92},
	PUBLISHER = {Cambridge University Press, Cambridge},
	YEAR = {1989},
	PAGES = {x+197},
	ISBN = {0-521-36136-2},
	MRCLASS = {35P15 (35J25 35P20 47F05 58G05 58G11 58G25)},
	MRNUMBER = {990239},
	MRREVIEWER = {H. Triebel},
	DOI = {10.1017/CBO9780511566158},
	URL = {https://doi.org/10.1017/CBO9780511566158},
}

@incollection {GrigorYan1999,
	AUTHOR = {Grigor'Yan, A.},
	TITLE = {Estimates of heat kernels on {R}iemannian manifolds},
	BOOKTITLE = {Spectral theory and geometry ({E}dinburgh, 1998)},
	SERIES = {London Math. Soc. Lecture Note Ser.},
	VOLUME = {273},
	PAGES = {140--225},
	PUBLISHER = {Cambridge Univ. Press, Cambridge},
	YEAR = {1999},
	MRCLASS = {58J35 (53C20)},
	MRNUMBER = {1736868},
	MRREVIEWER = {Thierry Coulhon},
	DOI = {10.1017/CBO9780511566165.008},
	URL = {https://doi.org/10.1017/CBO9780511566165.008},
}

@article {Caffarelli-Silvestre2007,
	AUTHOR = {Caffarelli, L. and Silvestre, L.},
	TITLE = {An extension problem related to the fractional {L}aplacian},
	JOURNAL = {Comm. Partial Differential Equations},
	FJOURNAL = {Communications in Partial Differential Equations},
	VOLUME = {32},
	YEAR = {2007},
	NUMBER = {7-9},
	PAGES = {1245--1260},
	ISSN = {0360-5302},
	MRCLASS = {35J70},
	MRNUMBER = {2354493},
	MRREVIEWER = {Francesco Petitta},
	DOI = {10.1080/03605300600987306},
	URL = {https://doi.org/10.1080/03605300600987306},
}

@article {Stinga-Torrea2010,
	AUTHOR = {Stinga, P. R. and Torrea, J. L.},
	TITLE = {Extension problem and {H}arnack's inequality for some
	fractional operators},
	JOURNAL = {Comm. Partial Differential Equations},
	FJOURNAL = {Communications in Partial Differential Equations},
	VOLUME = {35},
	YEAR = {2010},
	NUMBER = {11},
	PAGES = {2092--2122},
	ISSN = {0360-5302},
	MRCLASS = {35R11 (35B45 35B50 35B51 35B65)},
	MRNUMBER = {2754080},
	MRREVIEWER = {Nasser-eddine Tatar},
	DOI = {10.1080/03605301003735680},
	URL = {https://doi.org/10.1080/03605301003735680},
}

@article {Caffarelli-Stinga2016,
	AUTHOR = {Caffarelli, L. and Stinga, P. R.},
	TITLE = {Fractional elliptic equations, {C}accioppoli estimates and
	regularity},
	JOURNAL = {Ann. Inst. H. Poincar\'{e} C Anal. Non Lin\'{e}aire},
	FJOURNAL = {Annales de l'Institut Henri Poincar\'{e} C. Analyse Non Lin\'{e}aire},
	VOLUME = {33},
	YEAR = {2016},
	NUMBER = {3},
	PAGES = {767--807},
	ISSN = {0294-1449},
	MRCLASS = {35R11 (35B45 35B65 46E35)},
	MRNUMBER = {3489634},
	MRREVIEWER = {Mark Allen},
	DOI = {10.1016/j.anihpc.2015.01.004},
	URL = {https://doi.org/10.1016/j.anihpc.2015.01.004},
}

@book {Hebey1999,
	AUTHOR = {Hebey, E.},
	TITLE = {Nonlinear analysis on manifolds: {S}obolev spaces and
	inequalities},
	SERIES = {Courant Lecture Notes in Mathematics},
	VOLUME = {5},
	PUBLISHER = {New York University, Courant Institute of Mathematical
	Sciences, New York; American Mathematical Society, Providence,
	RI},
	YEAR = {1999},
	PAGES = {x+309},
	ISBN = {0-9658703-4-0; 0-8218-2700-6},
	MRCLASS = {58D15 (35J60 46E35 53C21 58J60)},
	MRNUMBER = {1688256},
	MRREVIEWER = {Gilles\ Carron},
}

@article {Aubin1976,
	AUTHOR = {Aubin, T.},
	TITLE = {\'Equations diff\'erentielles non lin\'eaires et probl\`eme de
	{Y}amabe concernant la courbure scalaire},
	JOURNAL = {J. Math. Pures Appl. (9)},
	FJOURNAL = {Journal de Math\'ematiques Pures et Appliqu\'ees. Neuvi\`eme
	S\'erie},
	VOLUME = {55},
	YEAR = {1976},
	NUMBER = {3},
	PAGES = {269--296},
	ISSN = {0021-7824,1776-3371},
	MRCLASS = {58G15 (53C20)},
	MRNUMBER = {431287},
	MRREVIEWER = {J.\ L.\ Kazdan},
}

@incollection {Bakry1994,
	AUTHOR = {Bakry, D.},
	TITLE = {L'hypercontractivit\'e{} et son utilisation en th\'eorie des
	semigroupes},
	BOOKTITLE = {Lectures on probability theory ({S}aint-{F}lour, 1992)},
	SERIES = {Lecture Notes in Math.},
	VOLUME = {1581},
	PAGES = {1--114},
	PUBLISHER = {Springer, Berlin},
	YEAR = {1994},
	ISBN = {3-540-58208-8},
	MRCLASS = {47D07 (47N30 58G11 60J25)},
	MRNUMBER = {1307413},
	MRREVIEWER = {Thierry\ Coulhon},
	DOI = {10.1007/BFb0073872},
	URL = {https://doi.org/10.1007/BFb0073872},
}

@article {Druet1998,
	AUTHOR = {Druet, O.},
	TITLE = {Optimal {S}obolev inequalities of arbitrary order on compact
	{R}iemannian manifolds},
	JOURNAL = {J. Funct. Anal.},
	FJOURNAL = {Journal of Functional Analysis},
	VOLUME = {159},
	YEAR = {1998},
	NUMBER = {1},
	PAGES = {217--242},
	ISSN = {0022-1236,1096-0783},
	MRCLASS = {53C21 (46E35)},
	MRNUMBER = {1654123},
	MRREVIEWER = {Gilles\ Carron},
	DOI = {10.1006/jfan.1998.3264},
	URL = {https://doi.org/10.1006/jfan.1998.3264},
}

@article {kreuml-Mordhorst2019,
	AUTHOR = {Kreuml, A. and Mordhorst, O.},
	TITLE = {Fractional {S}obolev norms and {BV} functions on manifolds},
	JOURNAL = {Nonlinear Anal.},
	FJOURNAL = {Nonlinear Analysis. Theory, Methods \& Applications. An
	International Multidisciplinary Journal},
	VOLUME = {187},
	YEAR = {2019},
	PAGES = {450--466},
	ISSN = {0362-546X,1873-5215},
	MRCLASS = {46E35 (42B35)},
	MRNUMBER = {3975112},
	MRREVIEWER = {Martin\ Lind},
	DOI = {10.1016/j.na.2019.06.014},
	URL = {https://doi.org/10.1016/j.na.2019.06.014},
}

@book {Chavel1993,
	AUTHOR = {Chavel, I.},
	TITLE = {Riemannian geometry---a modern introduction},
	SERIES = {Cambridge Tracts in Mathematics},
	VOLUME = {108},
	PUBLISHER = {Cambridge University Press, Cambridge},
	YEAR = {1993},
	PAGES = {xii+386},
	ISBN = {0-521-43201-4; 0-521-48578-9},
	MRCLASS = {53-02 (53Cxx)},
	MRNUMBER = {1271141},
	MRREVIEWER = {Carolyn\ Gordon},
}

@book {Hebey1996,
	AUTHOR = {Hebey, E.},
	TITLE = {Sobolev spaces on {R}iemannian manifolds},
	SERIES = {Lecture Notes in Mathematics},
	VOLUME = {1635},
	PUBLISHER = {Springer-Verlag, Berlin},
	YEAR = {1996},
	PAGES = {x+116},
	ISBN = {3-540-61722-1},
	MRCLASS = {46E35 (58D15)},
	MRNUMBER = {1481970},
	MRREVIEWER = {H.\ Triebel},
	DOI = {10.1007/BFb0092907},
	URL = {https://doi.org/10.1007/BFb0092907},
}

@book {DoCarmo1992,
	AUTHOR = {do Carmo, M. P.},
	TITLE = {Riemannian geometry},
	SERIES = {Mathematics: Theory \& Applications},
	EDITION = {Portuguese},
	PUBLISHER = {Birkh\"auser Boston, Inc., Boston, MA},
	YEAR = {1992},
	PAGES = {xiv+300},
	ISBN = {0-8176-3490-8},
	MRCLASS = {53-01},
	MRNUMBER = {1138207},
	MRREVIEWER = {Bang-yen\ Chen},
	DOI = {10.1007/978-1-4757-2201-7},
	URL = {https://doi.org/10.1007/978-1-4757-2201-7},
}

@book {Jost1995,
	AUTHOR = {Jost, J.},
	TITLE = {Riemannian geometry and geometric analysis},
	SERIES = {Universitext},
	PUBLISHER = {Springer-Verlag, Berlin},
	YEAR = {1995},
	PAGES = {xii+401},
	ISBN = {3-540-57113-2},
	MRCLASS = {53C21 (53-01 53C20 58-01 58E05 58E10)},
	MRNUMBER = {1351009},
	MRREVIEWER = {Man\ Chun\ Leung},
	DOI = {10.1007/978-3-662-03118-6},
	URL = {https://doi.org/10.1007/978-3-662-03118-6},
}

@article {Rey-Saintier2024,
    AUTHOR = {Rey, C.A. and Saintier, N.},
     TITLE = {Non-local equations and optimal {S}obolev inequalities on
              compact manifolds},
   JOURNAL = {J. Geom. Anal.},
  FJOURNAL = {Journal of Geometric Analysis},
    VOLUME = {34},
      YEAR = {2024},
    NUMBER = {1},
     PAGES = {Paper No. 17, 33},
      ISSN = {1050-6926,1559-002X},
   MRCLASS = {35R01 (26D10 35A15 35J92 35R11 58J05)},
  MRNUMBER = {4665057},
       DOI = {10.1007/s12220-023-01451-2},
       URL = {https://doi.org/10.1007/s12220-023-01451-2},
}

@article {Caffarelli-Roquejoffre-Savin2010,
	AUTHOR   = {Caffarelli, L. and Roquejoffre, J. and Savin, O.},
	TITLE   = {Nonlocal minimal surfaces},
	JOURNAL   = {Comm. Pure Appl. Math.},
	FJOURNAL   = {Communications on Pure and Applied Mathematics},
	VOLUME   = {63},
	NUMBER   = {9},
	YEAR   = {2010},
	PAGES   = {1111--1144},
	ISSN   = {0010-3640},
	MRCLASS   = {49Q05 (35R11 53A10)},
	MRNUMBER   = {2675483},
	DOI  = {10.1002/cpa.20331},
	URL  = {https://doi.org/10.1002/cpa.20331},
}

@book {Bucur-Valdinoci2016,
	AUTHOR   = {Bucur, C. and Valdinoci, E.},
	TITLE   = {Nonlocal Diffusion and Applications},
	SERIES   = {Lecture Notes of the Unione Matematica Italiana},
	VOLUME   = {20},
	PUBLISHER   = {Springer, Cham; Unione Matematica Italiana, Bologna},
	YEAR   = {2016},
	PAGES   = {xviii+155},
	ISBN   = {978-3-319-28738-6},
	MRCLASS   = {35R11 (45K05 47G20)},
	MRNUMBER   = {3469920},
	DOI  = {10.1007/978-3-319-28739-3},
	URL  = {https://doi.org/10.1007/978-3-319-28739-3},
}

@article {Metzler-Klafter2000,
	AUTHOR   = {Metzler, R. and Klafter, J.},
	TITLE   = {The random walk's guide to anomalous diffusion: a fractional dynamics approach},
	JOURNAL   = {Phys. Rep.},
	FJOURNAL   = {Physics Reports},
	VOLUME   = {339},
	NUMBER   = {1},
	YEAR   = {2000},
	PAGES   = {1--77},
	ISSN   = {0370-1573},
	MRCLASS   = {82C41 (60G52 82C70)},
	MRNUMBER   = {1809267},
	DOI  = {10.1016/S0370-1573(00)00070-3},
	URL  = {https://doi.org/10.1016/S0370-1573(00)00070-3}}

@article {Silling2000,
	AUTHOR   = {Silling, S.A.},
	TITLE   = {Reformulation of elasticity theory for discontinuities and long-range forces},
	JOURNAL   = {J. Mech. Phys. Solids},
	FJOURNAL   = {Journal of the Mechanics and Physics of Solids},
	VOLUME   = {48},
	NUMBER   = {1},
	YEAR   = {2000},
	PAGES   = {175--209},
	ISSN   = {0022-5096},
	MRCLASS   = {74B05 (74A25)},
	MRNUMBER   = {1732485},
	DOI  = {10.1016/S0022-5096(99)00029-0},
	URL  = {https://doi.org/10.1016/S0022-5096(99)00029-0},
}

@article {Laskin2002,
	AUTHOR   = {Laskin, N.},
	TITLE   = {Fractional Schr\"{o}dinger equation},
	JOURNAL   = {Phys. Rev. E},
	FJOURNAL   = {Physical Review E. Statistical, Nonlinear, and Soft Matter Physics},
	VOLUME   = {66},
	YEAR   = {2002},
	PAGES   = {056108, 7},
	ISSN   = {1539-3755},
	MRCLASS   = {81Q05 (47G20)},
	MRNUMBER   = {1942465},
	DOI  = {10.1103/PhysRevE.66.056108},
	URL  = {https://doi.org/10.1103/PhysRevE.66.056108},
}

@article {Caselli-Florit-Simon-Serra-2024,
    AUTHOR = {Caselli, M. and Florit-Simon, E. and Serra, J.},
     TITLE = {Fractional {S}obolev spaces on {R}iemannian manifolds},
   JOURNAL = {Math. Ann.},
  FJOURNAL = {Mathematische Annalen},
    VOLUME = {390},
      YEAR = {2024},
    NUMBER = {4},
     PAGES = {6249--6314},
      ISSN = {0025-5831,1432-1807},
   MRCLASS = {58J60 (31C12 46E35 58J35)},
  MRNUMBER = {4816133},
       DOI = {10.1007/s00208-024-02894-w},
       URL = {https://doi.org/10.1007/s00208-024-02894-w},
}

@article {Lions1984-1,
    AUTHOR = {Lions, P.-L.},
     TITLE = {The concentration-compactness principle in the calculus of
              variations. {T}he locally compact case. {I}},
   JOURNAL = {Ann. Inst. H. Poincar\'e{} Anal. Non Lin\'eaire},
  FJOURNAL = {Annales de l'Institut Henri Poincar\'e. Analyse Non
              Lin\'eaire},
    VOLUME = {1},
      YEAR = {1984},
    NUMBER = {2},
     PAGES = {109--145},
      ISSN = {0294-1449},
   MRCLASS = {49A50 (49A22)},
  MRNUMBER = {778970},
MRREVIEWER = {Gianfranco\ Bottaro},
       URL = {http://www.numdam.org/item?id=AIHPC_1984__1_2_109_0},
}

@article {Lions1984-2,
    AUTHOR = {Lions, P.-L.},
     TITLE = {The concentration-compactness principle in the calculus of
              variations. {T}he locally compact case. {II}},
   JOURNAL = {Ann. Inst. H. Poincar\'e{} Anal. Non Lin\'eaire},
  FJOURNAL = {Annales de l'Institut Henri Poincar\'e. Analyse Non
              Lin\'eaire},
    VOLUME = {1},
      YEAR = {1984},
    NUMBER = {4},
     PAGES = {223--283},
      ISSN = {0294-1449},
   MRCLASS = {49A50 (49A22)},
  MRNUMBER = {778974},
MRREVIEWER = {Gianfranco\ Bottaro},
       URL = {http://www.numdam.org/item?id=AIHPC_1984__1_4_223_0},
}

@article {Hang2022,
    AUTHOR = {Hang, F.},
     TITLE = {Aubin type almost sharp {M}oser-{T}rudinger inequality
              revisited},
   JOURNAL = {J. Geom. Anal.},
  FJOURNAL = {Journal of Geometric Analysis},
    VOLUME = {32},
      YEAR = {2022},
    NUMBER = {9},
     PAGES = {Paper No. 230, 40},
      ISSN = {1050-6926,1559-002X},
   MRCLASS = {58C05 (35A23)},
  MRNUMBER = {4447304},
MRREVIEWER = {Ol\'impio\ Hiroshi\ Miyagaki},
       DOI = {10.1007/s12220-022-00969-1},
       URL = {https://doi.org/10.1007/s12220-022-00969-1},
}

@article {Yan2023,
    AUTHOR = {Yan, Z.},
     TITLE = {Improved higher-order {S}obolev inequalities on {CR} sphere},
   JOURNAL = {J. Funct. Anal.},
  FJOURNAL = {Journal of Functional Analysis},
    VOLUME = {284},
      YEAR = {2023},
    NUMBER = {10},
     PAGES = {Paper No. 109890, 34},
      ISSN = {0022-1236,1096-0783},
   MRCLASS = {39B62 (32V20 32V40 35B38 46E35)},
  MRNUMBER = {4554743},
MRREVIEWER = {Song\ Ying\ Li},
       DOI = {10.1016/j.jfa.2023.109890},
       URL = {https://doi.org/10.1016/j.jfa.2023.109890},
}

@article {Zeitler2025,
    AUTHOR = {Zeitler, S.},
     TITLE = {A sharp higher order {S}obolev inequality on {R}iemannian
              manifolds},
   JOURNAL = {J. Funct. Anal.},
  FJOURNAL = {Journal of Functional Analysis},
    VOLUME = {289},
      YEAR = {2025},
    NUMBER = {6},
     PAGES = {Paper No. 111001, 45},
      ISSN = {0022-1236,1096-0783},
   MRCLASS = {58J60 (26D10 35A23 35R01 46E35)},
  MRNUMBER = {4897642},
       DOI = {10.1016/j.jfa.2025.111001},
       URL = {https://doi.org/10.1016/j.jfa.2025.111001},
}

\end{document}